\documentclass{article}

\usepackage{amsmath}
\usepackage{amsthm}
\usepackage{natbib}
\usepackage{graphicx} 
\usepackage{epstopdf}
\usepackage{subcaption}
\usepackage{amsfonts}
\usepackage{amssymb}
\usepackage{bm}
\usepackage{mathtools}
\usepackage{algorithm}
\usepackage{etoolbox}
\usepackage{algpseudocode}
\usepackage{hhline}

\algnewcommand\algorithmicinput{\textbf{Input:}}
\algnewcommand\Input{\item[\algorithmicinput]}
\algnewcommand\algorithmicoutput{\textbf{Output:}}
\algnewcommand\Output{\item[\algorithmicoutput]}
\algnewcommand\algorithmicbreak{\textbf{break}}
\algnewcommand\Break{\item[\algorithmicbreak]}

\newtheorem{theorem}{Theorem}
\newtheorem{remark}{Remark}
\newtheorem{lemma}{Lemma}

\renewcommand{\bm}{\mathbf}

\newcommand{\atom}{\boldsymbol \varphi}
\renewcommand{\dim}{p}
\newcommand{\sig}{\bm s}
\newcommand{\err}{\boldsymbol \xi}
\newcommand{\res}{\bm r}

\newcommand{\indset}{\Omega}
\newcommand{\coef}{x}
\newcommand{\coh}{\mu}
\newcommand{\norm}[1]{\left\lVert#1\right\rVert}
\newcommand{\abs}[1]{\left| #1\right| }
\newcommand{\innerp}[2]{\left\langle #1, #2 \right\rangle}

\newcommand{\E}{\textup{E}}

\newcommand{\dictm}{\bm\Phi}
\newcommand{\coefv}{\bm x}
\newcommand{\supp}{\operatorname{supp}}
\newcommand{\atnum}{N}
\newcommand{\slev}{m}

\newcommand{\atome}{\bm y}

\newcommand{\atomeh}{\atome^{\perp}}

\newcommand{\rr}{\bm w}
\newcommand{\sigeff}{\sigma_{\operatorname{eff}}(\slev,\coh)}
\newcommand{\Y}{\bm Y}
\newcommand{\V}{\bm V}
\newcommand{\vv}{\bm v}

\newcommand{\Aout}{A_{\text{out}}}
\newcommand{\Ain}{A_{\text{in}}}
\newcommand{\Bout}{B_{\text{out}}}
\newcommand{\Bin}{B_{\text{in}}}
\newcommand{\tmu}{\tilde{\coh}}
\newcommand{\xmin}{\coef_{\min}}
\newcommand{\R}{\mathbb{R}}

\usepackage{lipsum}

\newcommand\blfootnote[1]{%
	\begingroup
	\renewcommand\thefootnote{}\footnote{#1}%
	\addtocounter{footnote}{-1}%
	\endgroup
}

\begin{document}

\title{Tight Recovery Guarantees for Orthogonal Matching Pursuit Under Gaussian Noise}

\author{{
Chen Amiraz, Robert Krauthgamer and Boaz Nadler}\\
Weizmann Institute of Science\\
}

\maketitle

\begin{abstract}
{Orthogonal Matching pursuit (OMP) is a popular algorithm to estimate an unknown sparse vector from multiple linear measurements of it. Assuming exact sparsity and that the measurements are corrupted by additive Gaussian noise, the success of OMP is often formulated as exactly recovering the support of the sparse vector. 
Several authors derived a sufficient condition for exact support recovery by OMP with high probability depending on the signal-to-noise ratio, defined as the magnitude of the smallest non-zero coefficient of the vector divided by the noise level. 
We make two contributions. First, we derive a slightly sharper sufficient condition for two variants of OMP, in which either the sparsity level or the noise level is known. Next, we show that this sharper sufficient condition is tight, in the following sense: for a wide range of problem parameters, there exist a dictionary of linear measurements and a sparse vector with a signal-to-noise ratio slightly below that of the sufficient condition, for which with high probability OMP fails to recover its support. Finally, we present simulations which illustrate that our condition is tight for a much broader range of dictionaries.
}

{\textbf{\textit{Keywords:}} Compressed sensing,
	inverse problems,
	mutual incoherence,
	orthogonal matching pursuit (OMP),
	signal reconstruction,
	sparse estimation,
	support recovery.}
\blfootnote{This article has been accepted for publication in Information and Inference: A Journal of the IMA, published by Oxford University Press.}
\end{abstract}

\section{Introduction}
A fundamental inverse problem arising in a wide variety of fields is to estimate an unknown sparse vector $\coefv\in\mathbb{R}^{\atnum}$ from $\dim$ linear measurements of it, often with $\dim<\atnum$.
Notable examples in signal processing include sparse recovery in a redundant representation and compressed sensing \citep{elad2010sparse,foucart2013mathematical}. 
A notable example in statistics is linear regression with a sparse coefficient vector, in particular when there are more variables than observations \citep{tibshirani2015statistical}.

Assuming that the measurements are corrupted by additive Gaussian noise, the observed signal $\sig\in\R^{\dim}$ has the following form
\begin{equation}\label{eq:sig_dec}
\sig = \dictm \coefv+\sigma\err
\end{equation}
where $\dictm \in \R^{\dim\times\atnum}$ is a known overcomplete matrix, $\coefv\in\R^{\atnum}$ is an unknown sparse vector, $\err\in\R ^\dim$ is a random Gaussian noise vector $\err\sim N\left(\bm 0,\bm I_{\dim}\right)$ and $\sigma>0$ is the noise level.
We say that $\coefv$ is \emph{$\slev$-sparse} if $\norm{\coefv}_0=\abs{\supp(\coefv)}=m$ and denote its support by $\indset = \supp(\coefv)$.    
In statistics $\dictm$ is referred to as the design matrix, whereas in the signal processing literature it is often called the \emph{dictionary}. We refer to the columns $\atom_i$ of $\dictm$ as the \emph{atoms} of the dictionary and assume for simplicity that they are normalized to have unit norm $\norm{\atom_i} = 1$.

In sparse recovery, the observed signal $\sig$, the dictionary $\dictm$ and the sparsity level $\slev$ are given as input, and the goal is to output an estimate $\hat{\coefv}$ that is close to the unknown vector $\coefv$. Under the assumption that $\err$ is Gaussian and independent of $\coefv$, the maximum likelihood solution is
\begin{equation}\label{eq:SR_problem}
\hat{\coefv} = \arg\min \left\lbrace \norm{\sig-\dictm\bm{z}}_2: \norm{\bm{z}}_0\leq\slev\right\rbrace .
\end{equation}

In the noiseless case $\sigma=0$, minimizing \eqref{eq:SR_problem} is equivalent to finding an $\slev$-sparse vector $\hat{\coefv}$ such that $\sig = \dictm \hat{\coefv}$.
For $\dim<\atnum$ this linear system is underdetermined and 
may have multiple solutions. Hence, for any $\sigma\geq 0$, Eq. \eqref{eq:SR_problem} may in general also have multiple solutions.
In certain regimes there exists a unique solution, for example when $\slev$ is small compared to the size of the smallest linearly-dependent subset of dictionary atoms \citep{donoho2003optimally}.
Furthermore, even if a unique solution exists, finding it is in general NP-hard because the sparsity constraint is non-convex \citep{davis1997adaptive}.  
Over the last decades, several polynomial-time methods were developed for estimating $\hat{\coefv}$.
Convex optimization-based methods such as Basis Pursuit use a relaxation of the $l_0$-norm of $\coefv$ to its $l_1$-norm \citep{tibshirani1996regression,chen2001atomic}. 
Other recovery methods use non-convex penalty functions that promote sparsity \citep{chartrand2008iteratively,daubechies2010iteratively,figueiredo2007majorization}. 
Greedy methods estimate $\coefv$ by iteratively selecting atoms that have high correlation with the residual part of the signal \citep{dai2009subspace,needell2009cosamp,needell2010signal}.
For a recent review of sparse recovery algorithms, see \citep{marques2019review} and the references therein.

In this work, we focus on Orthogonal Matching Pursuit (OMP), described in Algorithm \ref{alg:OMP}, which is one of the simplest and fastest greedy methods for sparse recovery \citep{chen1989orthogonal,pati1993orthogonal,mallat1993matching}.
One key challenge in OMP computing an estimate $\hat{\coefv}$ close to $\coefv$ is to accurately estimate its support.
Hence, several authors studied conditions under which OMP exactly recovers the support of $\coefv$.

\begin{algorithm}[H]
	\caption{OMP}
	\label{alg:OMP}
	\textbf{Input}  dictionary $\dictm\in\mathbb{R}^{\dim\times \atnum}$, signal $\sig\in\mathbb{R}^\dim$, sparsity level $\slev$\\
	\textbf{Output} estimated vector $\hat{\coefv}_\slev\in\mathbb{R}^{ \atnum}$
	\begin{algorithmic}[1]          
		\State initialize the residual $\res_0=\sig$ and the estimated support $\hat{\indset}_0=\emptyset$
		\For{$t=1,\dots,\slev$} 
		\State  calculate $j=\arg\max\left\lbrace  \lvert \langle \atom_i,\res_{t-1}\rangle \rvert: i\in [\atnum]\right\rbrace $ 
		\State  add $\hat{\indset}_t=\hat{\indset}_{t-1}\cup\{j\}$
		\State  calculate $\hat{\coefv}_t= \arg\min \left\lbrace \left\Vert \sig-\dictm\coefv\right\Vert_2 :\coefv\in\mathbb{R}^{\atnum},\supp(\coefv)=\hat{\indset}_{t}\right\rbrace  $
		\State  update $\res_t=\sig-\dictm \hat{\coefv}_t$
		\EndFor
	\end{algorithmic}
\end{algorithm}

Several conditions for exact support recovery by OMP and by other methods have been studied. These include the Restricted Isometry Property (RIP) \citep{candes2005decoding}, the Exact Recovery Condition (ERC) \citep{tropp2004greed} and the Mutual Incoherence Property (MIP) \citep{donoho2001uncertainty}. 
For RIP and ERC based guarantees, see \citep{cai2018improved,hashemi2016sparse} and the references therein.
While MIP is more restrictive than the other conditions, it is simple and tractable to compute for arbitrary dictionaries. In this work we thus restrict our attention to coherence-based guarantees.
Specifically, the coherence of the dictionary $\dictm$ is defined as 
\begin{equation}\label{eq:coh}
\coh=\coh\left( \dictm\right) =\max_{i\neq j}\left|\innerp{\atom_{i}}{\atom_{j}}\right|.
\end{equation}
An $\slev$-sparse vector $\coefv$ satisfies the Mutual Incoherence Property (MIP) if
\begin{equation}
\label{eq:mip_cond}
\coh<\frac{1}{2\slev-1}.
\end{equation}
A fundamental result by \citet{tropp2004greed} is that the MIP condition is sufficient for exact support recovery by OMP in the noiseless case.
\citet{cai2010stable} proved that the MIP condition is sharp in the following setting: 
for each pair of positive integers $\left( \slev,k\right) $, there exist a dictionary of size $ 2\slev k \times \left( 2\slev-1\right)  k$ with coherence $\coh=\frac{1}{2\slev-1}$ and an $\slev$-sparse vector such that OMP fails to recover its support.

In the presence of additive Gaussian noise with noise level $\sigma>0$, even if an \(m\)-sparse vector \(\coefv\) satisfies the MIP condition \eqref{eq:mip_cond},
its exact support recovery will depend on the specific noise realization in the observed signal $\sig$. Hence, exact support recovery can only be guaranteed with a success probability $P_{\operatorname{succ}}<1$, which in general depends on the noise level $\sigma$, the sparsity level $\slev$, the magnitude of the non-zero coefficients of $\coefv$, the dictionary dimensions $\dim$ and $\atnum$ and the coherence $\coh$.
As we review in Section \ref{sec:results}, \citet{ben2010coherence} developed a sufficient condition for OMP to recover the support of $\coefv$ in the presence of additive Gaussian noise with high probability. A similar result for a variant of OMP was proved by \citet{cai2011orthogonal}.
\citet{miandji2017probability} derive a similar sufficient condition in a different model where the nonzero elements of $\coefv$ are random variables.

In this paper we make two key contributions. First, in Theorem \ref{thm:sharp_BH} we derive a sharper sufficient condition than that of \citet{ben2010coherence} and \citet{cai2011orthogonal} by performing a tighter analysis of their proof. An interesting question is whether this sufficient condition is sharp, or can it be lowered further.
Our main result, stated formally in Theorem \ref{thm:LB}, shows that this sharper sufficient condition is quite tight. Specifically, for a wide range of sparsity levels $\slev$, dictionary dimensions $\dim$, $\atnum$ and coherence values $\coh$, there exist a dictionary $\dictm$ and a vector $\coefv$ with a signal-to-noise ratio that is slightly lower than that of our sufficient condition, for which with high probability OMP fails to recover its support.
In Section \ref{sec:sim} we present several simulations that support our theoretical analysis.
All proofs can be found in Section \ref{sec:proofs}.

        \section{Main Results}\label{sec:results}
We first introduce some notation. We denote $\xmin=\min_{i\in\indset}\abs{\coefv_{i}}$ and define the following effective noise factor \[\sigeff = \frac{\sigma}{1-\left(2\slev-1\right)\coh}. \]
Throughout the paper we assume that the MIP condition \eqref{eq:mip_cond} holds, so $\sigeff$ is well defined and strictly positive.

For measurements that are corrupted by additive Gaussian noise, \citet{ben2010coherence} derived the following sufficient condition for OMP to recover the support of $\coefv$ with high probability.
\begin{theorem}[\citet{ben2010coherence}]\label{thm:BenHaim}
	Let $\coefv$ be an unknown vector with known sparsity $\slev$, and let $\sig = \dictm \coefv+\sigma\err$, where $\dictm\in \mathbb{R}^{\dim \times \atnum} $ is a dictionary with normalized columns and coherence $\coh$, and $\err\sim N\left(\bm 0,\bm I_{\dim}\right)$.
	Suppose that the MIP condition \eqref{eq:mip_cond} holds and that for some $\alpha\geq 0$
	\begin{equation}\label{eq:SNR_upper}
	\xmin\geq 2\sigeff  \sqrt{2\left(1+\alpha \right) \log \atnum}.
	\end{equation}
	Then, OMP with $\slev$ iterations successfully recovers the support of $\coefv$ with probability at least 
	\begin{equation}\label{eq:BH_prob}
	1-\frac{1}{\atnum^\alpha \sqrt{\pi\left(1+\alpha \right)\log \atnum}}.
	\end{equation}
\end{theorem}

In many practical cases $\slev$ is unknown while the noise level $\sigma$ is known. Denote by OMP* a variant of Algorithm \ref{alg:OMP} where instead of performing $\slev$ iterations, the algorithm stops when the maximal correlation of the residual with any dictionary atom is smaller than a threshold $\tau$, i.e., $\norm{\dictm^T\res_{t}}_{\infty}\leq\tau$.
\citet[Thm. 8]{cai2011orthogonal} proved the following analogue of Theorem \ref{thm:BenHaim}. Under the MIP condition \eqref{eq:mip_cond} and the same condition \eqref{eq:SNR_upper}, OMP* with threshold $\tau= \sigma \sqrt{2(1+\alpha)\log \atnum}$ recovers the support of $\coefv$ with probability at least $1-\slev/\atnum^\alpha\sqrt{2\log \atnum}$.

\subsection{Sharper sufficient condition}\label{sec:suff}
By performing a tighter analysis of the proofs of \citet{ben2010coherence} and \citet{cai2011orthogonal}, we derive a sharper sufficient condition than \eqref{eq:SNR_upper} for exact support recovery by both OMP and OMP*.
However, this sharper sufficient condition comes at a price, whereby the success probability is a function not only of the vector length $\atnum$, but also of its sparsity level $\slev$. 
The following theorem formalizes this statement and is proved in Section \ref{sec:proof_thm_sharp_BH}.
\begin{theorem}\label{thm:sharp_BH}
	Let $\coefv$ be an unknown vector with known sparsity $\slev$, and let $\sig = \dictm \coefv+\sigma\err$, where $\dictm\in \mathbb{R}^{\dim \times \atnum} $ is a dictionary with normalized columns and coherence $\coh$, and $\err\sim N\left(\bm 0,\bm I_{\dim}\right)$.
	Suppose that the MIP condition \eqref{eq:mip_cond} holds, that $\slev\leq\atnum^\beta$ for some $0< \beta < 1$ and that for some $\alpha\geq 0$
	\begin{equation}\label{eq:SNR_upper_sharp}
	\xmin\geq\sigeff\left(1+\sqrt{\beta}\right)\sqrt{2\left(1+\alpha\right)\log\atnum}.
	\end{equation}
	Then, OMP with $\slev$ iterations successfully recovers the support of $\coefv$ with probability at least 
	\begin{equation}\label{eq:sharp_pr}
	1-\frac{1}{\sqrt{\pi\left(1+\alpha\right)\log \atnum}}\left(\frac{1}{ \atnum^{\alpha}}+\frac{1}{ \atnum^{\alpha\beta}\sqrt{\beta}}\right).
	\end{equation}
	Moreover, under the same conditions OMP* with threshold $\tau=\sigma \sqrt{2(1+\alpha)\log \atnum}$ successfully recovers the support of $\coefv$ with probability at least \eqref{eq:sharp_pr}.
\end{theorem}
Eq. \eqref{eq:SNR_upper_sharp} is sharper than Eq. \eqref{eq:SNR_upper} since $\beta<1$.
Simulations in Section \ref{sec:sim} illustrate the tightness of this result.

\subsection{Near-tightness of the OMP recovery guarantee}\label{sec:necessary}

According to either Eq. \eqref{eq:BH_prob} or \eqref{eq:sharp_pr}, the smallest $\alpha$ that still guarantees exact support recovery with probability tending to $1$ as $\atnum\rightarrow\infty$ is $\alpha=0$. Therefore, the weakest sufficient condition for OMP to recover the exact support of $\coefv$ with high probability for $N\gg 1$ is
\begin{equation}\label{eq:apprx_ub}
\xmin\geq \sigeff \left(1+\sqrt{\beta}\right) \sqrt{2\log \atnum}.
\end{equation}
An interesting question is thus whether this sufficient condition is sharp, or could the right hand side in \eqref{eq:apprx_ub} be lowered further. 

The main result of this paper, formalized in Theorem \ref{thm:LB} below, is that the above condition is quite tight. 
Informally, our result can be stated as follows: for a wide range of sparsity levels $\slev$, dictionary dimensions $\dim$, $\atnum$ and coherence values $\coh$, there exist a dictionary $\dictm\in\R^{\dim\times\atnum}$ and an $\slev$-sparse vector $\coefv\in\R^{\atnum}$ with
\begin{equation}\label{eq:apprx_lb}
\xmin\approx\sigeff\left( 1-\coh-\sqrt{\beta}\right) \sqrt{2\log\atnum},
\end{equation}
for which OMP fails to recover its support with probability $1-o(1)$. In particular, the failure probability for this specific $ \dictm$ and $\coefv$ tends to $1$ as $\atnum\rightarrow\infty$.
As shown by the simulations in Section \ref{sec:sim}, OMP fails with high probability under condition \eqref{eq:apprx_lb} in a much broader range of cases.
These include a case where the dictionary atoms are drawn independently and uniformly at random from the unit sphere and a case where the dictionary is composed of two orthogonal matrices (the identity matrix and the Hadamard matrix with normalized columns).

If $\slev$ is constant or polylogarithmic in $\atnum$, then as $\atnum\rightarrow\infty$ we can take $\beta>0$ arbitrarily small. In this case, the bounds \eqref{eq:apprx_ub} and \eqref{eq:apprx_lb} match, up to a multiplicative factor of $1-\coh$.
Finally, for various dictionaries the coherence $\coh$ is itself small. For example, if each entry of the dictionary is drawn independently and uniformly at random from $\pm 1/\sqrt{\dim} $, then with probability exceeding $1-\delta^2$ the coherence is $\coh\leq 2\sqrt{\dim^{-1}\log\frac{\atnum}{\delta}}$ \citep{tropp2007signal}.
Hence, $\coh\rightarrow 0$ if $\atnum$ is sub-exponential in $\dim$.

To formally state our theorem, we introduce the following notations.   
First, let
\begin{equation}\label{eq:eps}
\rho = \rho\left(\slev,\coh \right)  = \sqrt{\frac{1-(\slev-1)\coh}{\slev}}
\end{equation}
and 
\begin{equation}\label{eq:tmu}
\tmu  = \tmu\left(\slev,\coh \right) = \frac{\coh^2}{\rho^2} = \frac{\coh^2\slev}{1-(\slev-1)\coh}.
\end{equation}  
Both quantities are well defined, since by the MIP condition \eqref{eq:mip_cond}, $1-(\slev-1)\coh>0$. It can be easily shown that $\sqrt{\coh}<\rho\leq 1/\sqrt{\slev}$ and $\tmu<\coh$.
Next, denote $\tilde{\dim}=\dim-\slev$ and $\tilde{\atnum}=\atnum-\slev$.
Let $\coh_{\min}\left( a,b\right) $ be the smallest possible coherence of an $a\times b$ overcomplete dictionary with $a<b$.
To prove our theorem we construct a dictionary that consists of several parts. One of these parts is a $\tilde{\dim}\times \tilde{\atnum}$ dictionary with coherence $L = L\left( \tilde{\dim},\tilde{\atnum}\right) =\coh_{\min}\left( \tilde{\dim},\tilde{\atnum}\right)$. By the theory of Grassmannian frames,  $L\geq\sqrt{\frac{\tilde{\atnum}-\tilde{\dim}}{\tilde{\dim}(\tilde{\atnum}-1)}}$ \citep[see for example][]{strohmer2003grassmannian}.      
In fact, $L$ may be strictly higher since Grassmannian frames do not exist for every pair $(\tilde{\dim},\tilde{\atnum})$. 
However it can not be much higher, since by \citet{tropp2007signal} $L\leq 2\sqrt{\tilde{\dim}^{-1}\log \tilde{\atnum}}$.

We now give a rigorous statement of our result, whose proof appears in Section \ref{sec:proof_thm_LB}.
\begin{theorem}\label{thm:LB}
	Let $\dim,\atnum$ be integers such that $\dim < \atnum $. Let $\slev$ be an integer and let $\coh$ be a number that satisfy the MIP condition \eqref{eq:mip_cond} and the following set of inequalities:
	\begin{equation}\label{eq:sparsity_beta}
	\slev \leq \min\left\lbrace \atnum^\beta, \dim \right\rbrace 
	\end{equation}          
	where $0< \beta < 1$,
	\begin{equation}\label{eq:sparsity_cond}
	\slev\leq \frac{3-L-\sqrt{8-8L}}{L},
	\end{equation}
	and
	\begin{equation}\label{eq:mu_cond} 
	\frac{\left(1+L\left(\slev-1\right)\right)\left(1-\sqrt{1-\frac{4L\left(2\slev-1-L\slev\right)}{\left(1+L\left(\slev-1\right)\right)^{2}}}\right)}{2\left(2\slev-1-L\slev\right)}\leq\coh\leq\frac{\left(1+L\left(\slev-1\right)\right)\left(1+\sqrt{1-\frac{4L\left(2\slev-1-L\slev\right)}{\left(1+L\left(\slev-1\right)\right)^{2}}}\right)}{2\left(2\slev-1-L\slev\right)}.
	\end{equation}
	Then, there exists a dictionary $\dictm\in \R^{\dim\times\atnum}$ with coherence $\coh$ and a corresponding $\slev$-sparse vector $\coefv\in \R^{\atnum}$ satisfying 
	\begin{eqnarray}
	\label{eq:SNR_Cond}
	\xmin =  \sigeff  \cdot 
	\Bigg\{  \sqrt{2(1-\coh)(1-\tmu)\log\tilde{\atnum}}-\sqrt{2\beta\left( 1-\rho^{2}\right) \log\atnum} \nonumber \\
	-c_{0}\sqrt{(1-\tmu)\log\log\tilde{\atnum}}-\left( \rho+\sqrt{\tmu}\right) \sqrt{2\log\log\atnum}\Bigg\}
	\end{eqnarray}
	where $c_0>0$ is a universal constant, such that with probability at least \[P_0 = 1-6e^{-C\sqrt{\log\log\tilde{\atnum}\min\left\lbrace \coh^{-1},\;\log\tilde{\atnum}\right\rbrace }}-\left( {\log\atnum\sqrt{\pi\log\log\atnum}}\right) ^{-1}-\left( {\sqrt{\pi\beta\log \atnum}}\right) ^{-1},\] OMP fails to recover the support of $\coefv$ from $\sig = \dictm \coefv+\sigma\err$.
\end{theorem}

\begin{remark}\label{rem:mu_cond} 
	Let us now illustrate that conditions \eqref{eq:sparsity_cond} and \eqref{eq:mu_cond} are not very restrictive.
	It is instructive to consider the over-complete case with $ \atnum= J \dim$ for $J>1$, with $\atnum,\dim \gg 1$ and sparsity $\slev$ much smaller than $\dim$, such that $\tilde{\dim}\approx\dim$. By the theory of Grassmannian frames $L\approx \frac{C(J)}{\sqrt{\dim}}$ for an appropriate $C(J)>0$ \citep{strohmer2003grassmannian}. Under the MIP condition \eqref{eq:mip_cond}, $\slev \lesssim 0.5\frac{\sqrt{\dim}}{C(J)}$, while under condition \eqref{eq:sparsity_cond}, $\slev \lesssim (3-\sqrt{8})\frac{\sqrt{\dim}}{C(J)}\approx 0.17 \frac{\sqrt{\dim}}{C(J)}$.
	
	For values of $\slev$ such that $ L\slev$ is much smaller than $1$, condition \eqref{eq:mu_cond} can be approximated by a binomial approximation $\sqrt{1-\varepsilon}\approx 1-\varepsilon/2$ for small $\varepsilon = \frac{4L\left(2\slev-1-L\slev\right)}{\left(1+L\left(\slev-1\right)\right)^{2}}$ as
	\[L \lesssim\coh\lesssim\frac{1}{2\slev-1}-L.\]
	The inequality $L \lesssim\coh$ follows essentially from frame lower bounds whereas the other inequality is very close to the MIP condition \eqref{eq:mip_cond}.
	Hence, condition \eqref{eq:mu_cond} is only slightly more restrictive than MIP.
	This comparison is visualized in Figure \ref{fig:m_mu}.
\end{remark}
\begin{figure}[t]
	\centering
	\begin{subfigure}[b]{0.45\linewidth}
		\includegraphics[width=\linewidth]{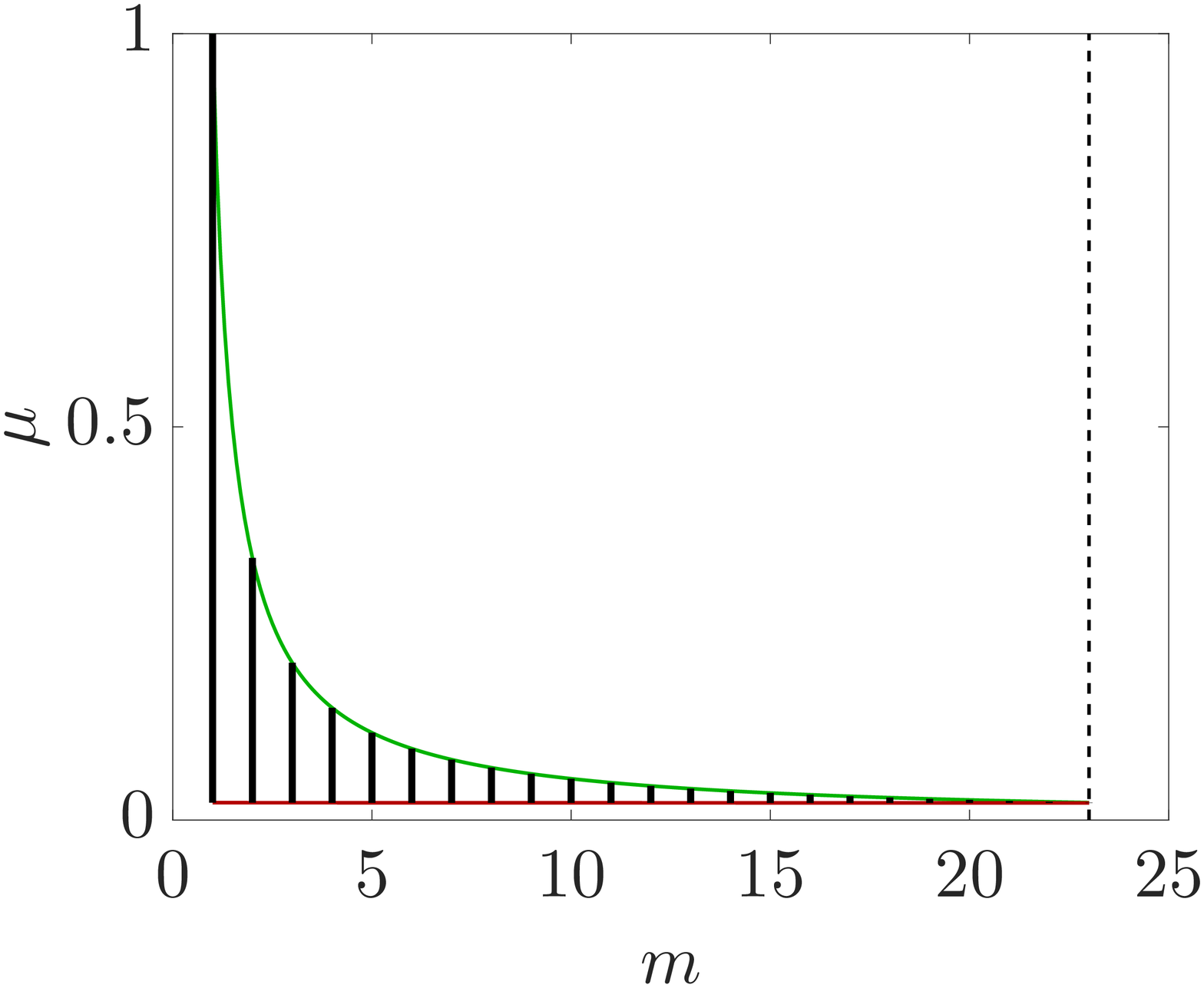}
		\caption{$\left( \slev,\coh\right) $ by the MIP condition \eqref{eq:mip_cond}}
	\end{subfigure}
	\begin{subfigure}[b]{0.45\linewidth}
		\includegraphics[width=\linewidth]{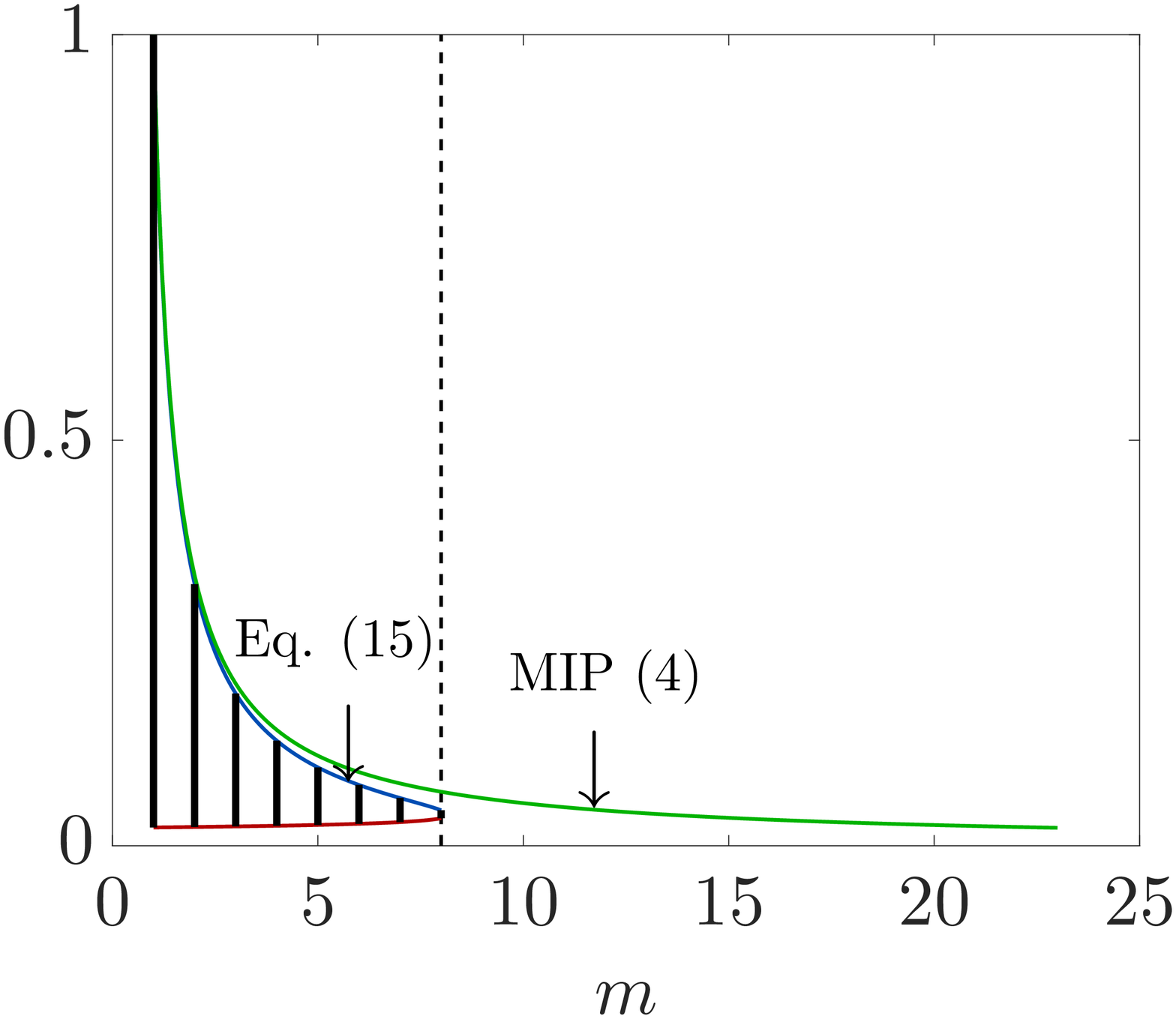}
		\caption{$\left( \slev,\coh\right) $ by condition \eqref{eq:mu_cond}}
	\end{subfigure}
	\caption{Example of requirements \eqref{eq:sparsity_cond} and \eqref{eq:mu_cond} on $\slev$ and $\coh$. 
		We set $\tilde{\dim}=1020$ and $\tilde{\atnum}=2040$, such that $L=\frac{1}{\sqrt{2039}}$ ($2039$ is a prime number), so different values of $\slev$ correspond to different values of $\dim$ and $\atnum$. 
		In the left panel, the solid black lines are the allowed values of $\slev$ and $\coh$ under the MIP condition \eqref{eq:mip_cond} and frame lower bounds.
		In the right panel, the solid black lines are the allowed values of $\slev$ and $\coh$ under \eqref{eq:mu_cond}. Condition \eqref{eq:sparsity_cond} is that the integer $\slev$ is to the left of the black dashed line.
		Note that the ratio between the largest possible $\slev$ value in each panel is approximately $\frac{0.5}{3-\sqrt{8}}\approx 2.9$.
	}
	\label{fig:m_mu}
\end{figure}
\begin{remark}
	We now show how Eq. \eqref{eq:SNR_Cond} may be approximated by Eq. \eqref{eq:apprx_lb}.
	First, for Theorem \ref{thm:LB} to be meaningful, the right hand side of Eq. \eqref{eq:SNR_Cond} must be positive. We now show that this is indeed the case for typical parameter values. If $\slev=\atnum^\beta$ for $\beta<1$, then  $\log\tilde{\atnum}=\log\atnum+\log\left( 1-\frac{1}{\atnum^{1-\beta}}\right) \approx\log\atnum$. Recall that $\tmu<\coh$ and that $\rho>0$. Hence, the first two terms on the right hand side of \eqref{eq:SNR_Cond} can be approximated as follows 
	\begin{eqnarray*}
		\sqrt{2(1-\coh)(1-\tmu)\log\tilde{\atnum}}-\sqrt{2\beta\left(1-\rho^{2}\right)\log\atnum} 
		& > & (1-\coh)\sqrt{2\log\tilde{\atnum}}-\sqrt{2\beta\log\atnum}\\
		& \approx & (1-\coh-\sqrt{\beta})\sqrt{2\log\atnum}.
	\end{eqnarray*}
	In addition, the last two terms on the right hand side of equation \eqref{eq:SNR_Cond} are small compared to the first term, since they are of order $\sqrt{\log\log\atnum}$.
	Hence, \eqref{eq:SNR_Cond} may be approximated by \eqref{eq:apprx_lb}.
\end{remark}

\section{Simulations}\label{sec:sim}
\begin{figure}[t]
	\centering
	\begin{subfigure}[b]{0.33\linewidth}
		\includegraphics[width=\linewidth]{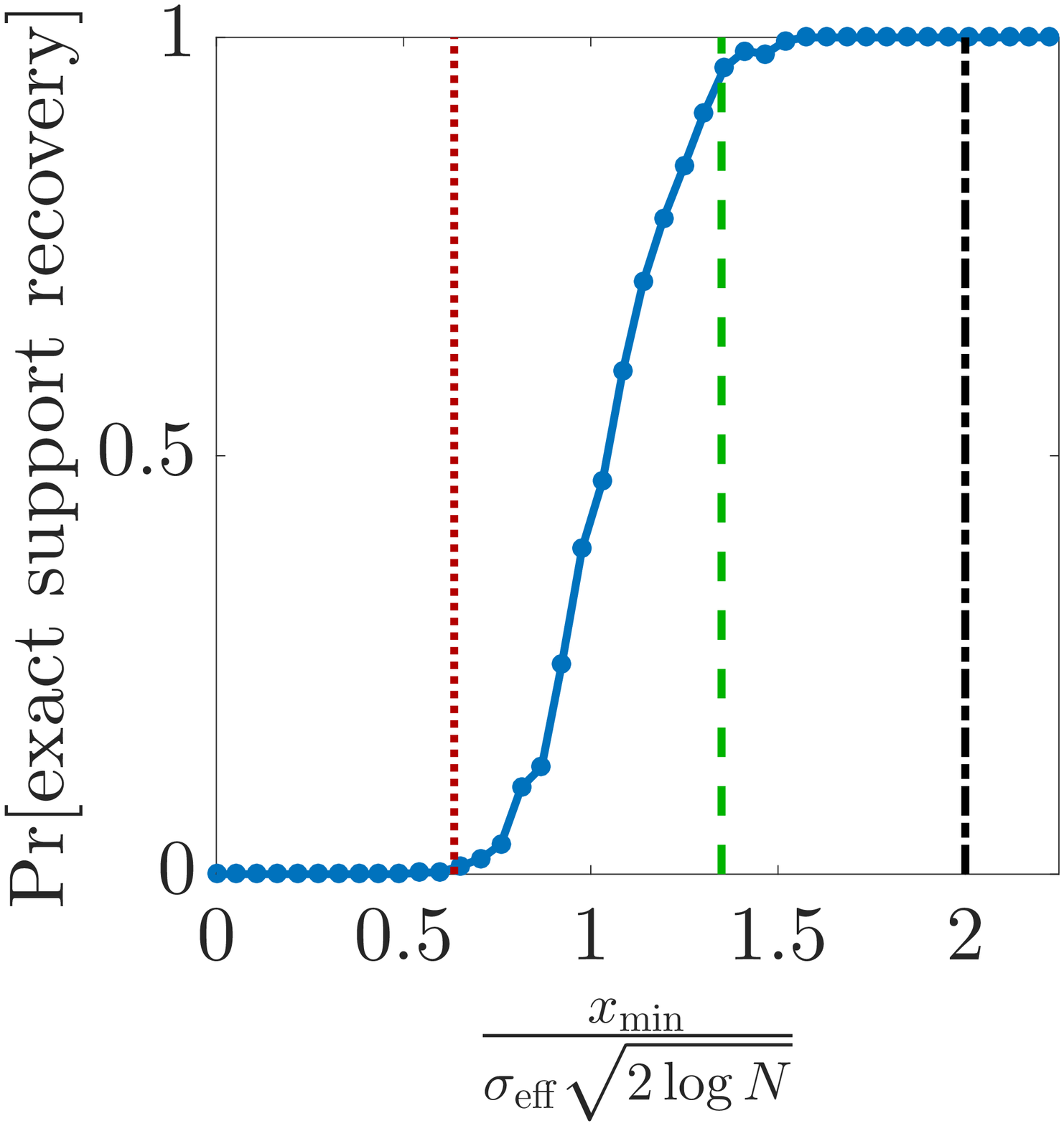}
		\caption{Two-ortho dictionary}
	\end{subfigure}
	\begin{subfigure}[b]{0.3\linewidth}
		\includegraphics[width=\linewidth]{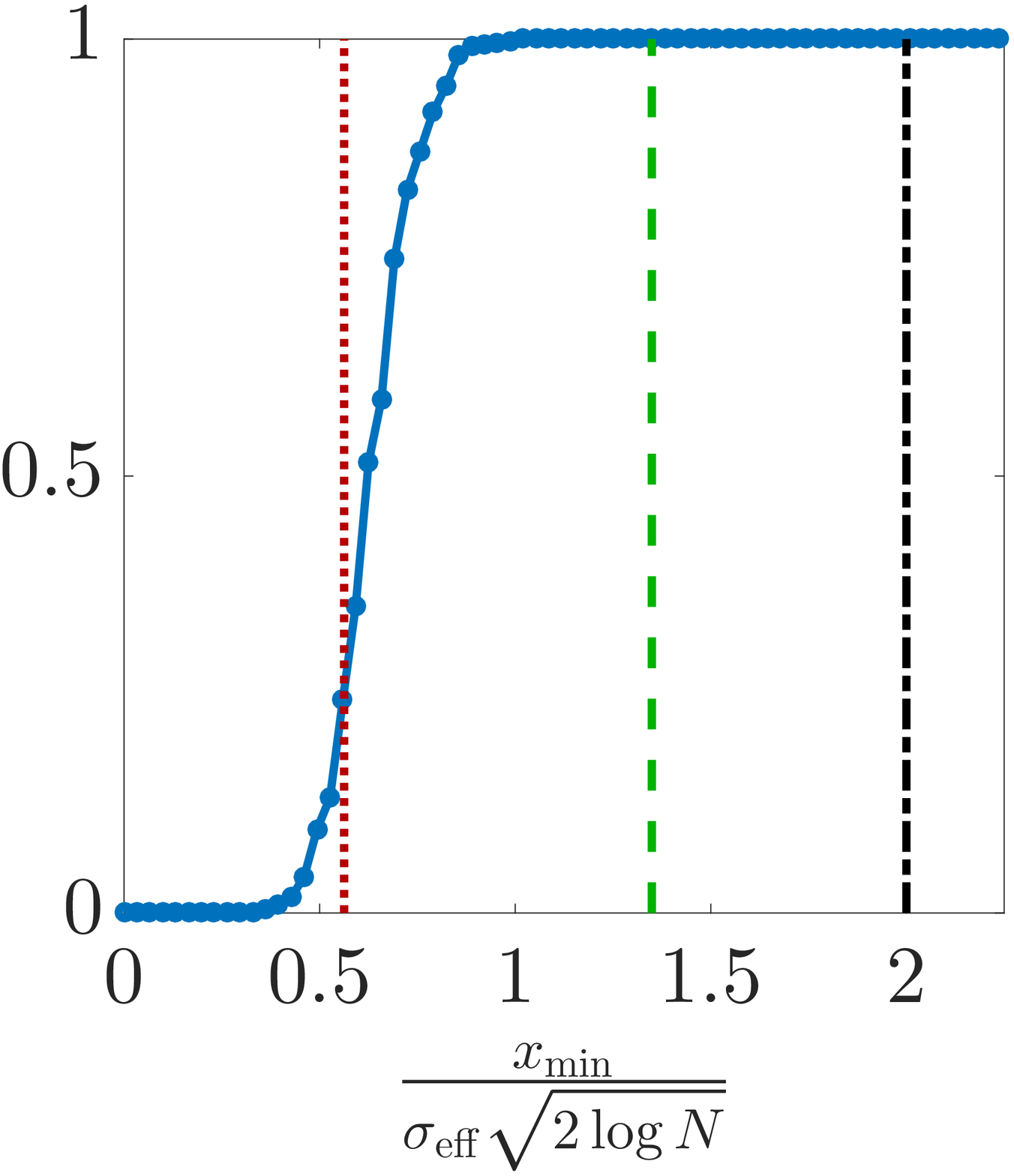}
		\caption{Random dictionary}
	\end{subfigure}
	\begin{subfigure}[b]{0.3\linewidth}
		\includegraphics[width=\linewidth]{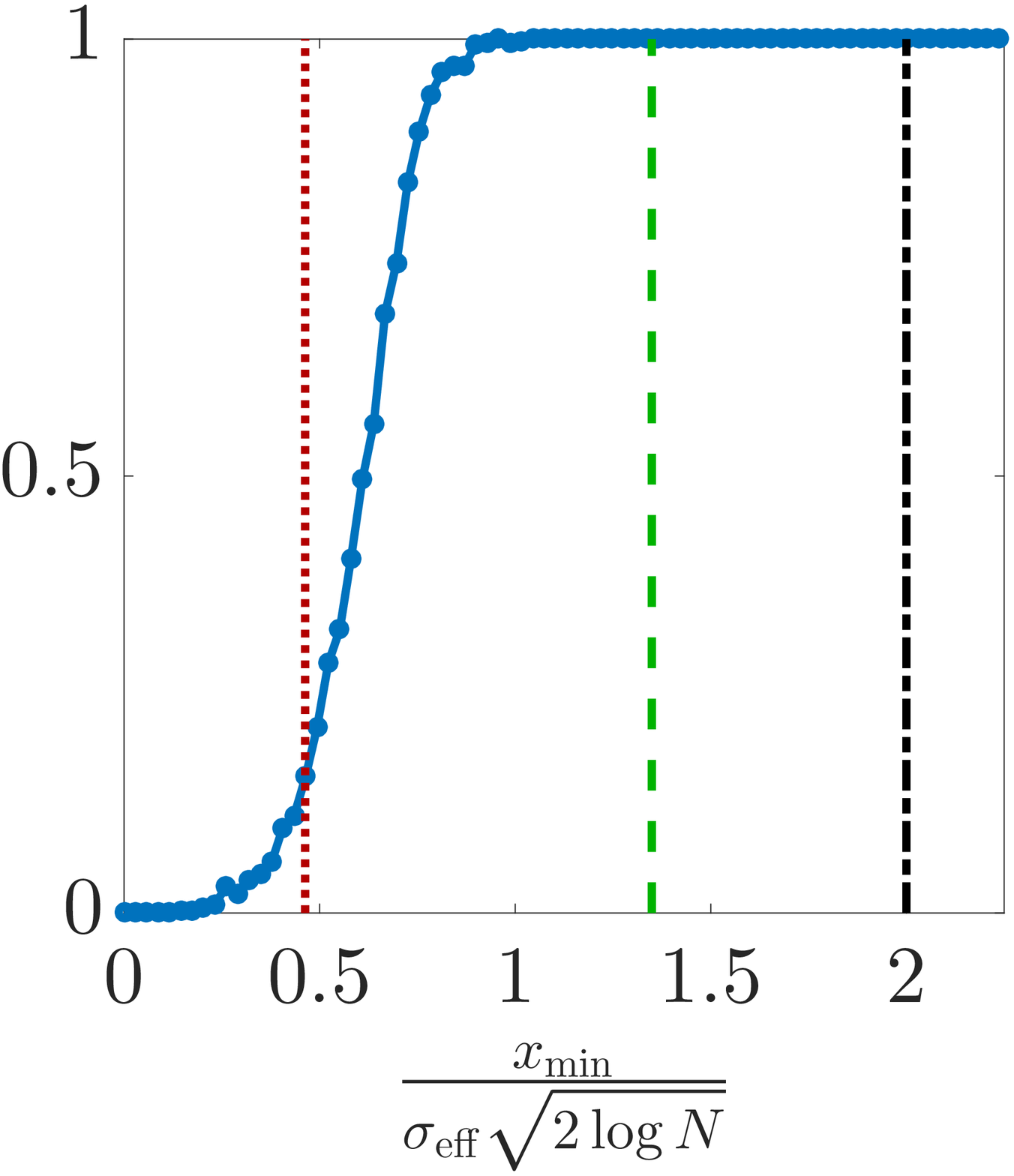}
		\caption{Our dictionary}
	\end{subfigure}
	\caption{The solid blue line in each panel is the empirical probability of exact support recovery of a sparse vector by OMP as a function of its normalized signal-to-noise ratio in Setting 1. The dash-dotted black line is the sufficient condition \eqref{eq:SNR_upper} by \citet{ben2010coherence}. The dashed green line is the sharper sufficient condition \eqref{eq:SNR_upper_sharp}. The dotted red line is the approximate condition \eqref{eq:apprx_lb}. 
	}
	\label{fig:differnt_dict}
\end{figure}
\begin{figure}[t]
	\centering
	\begin{subfigure}[b]{0.33\linewidth}
		\includegraphics[width=\linewidth]{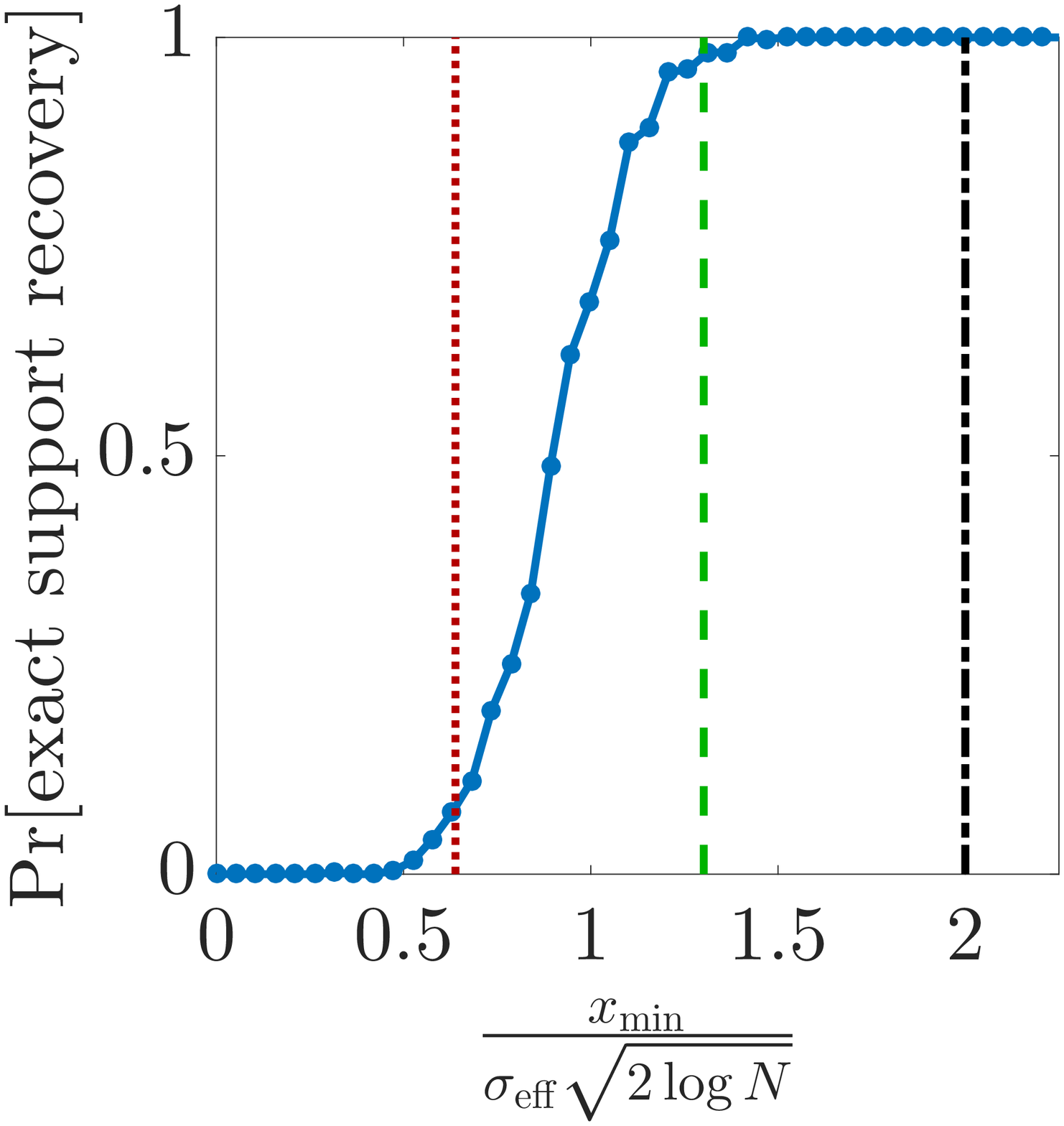}
		\caption{Sparsity $\slev=2$}
	\end{subfigure}
	\begin{subfigure}[b]{0.3\linewidth}
		\includegraphics[width=\linewidth]{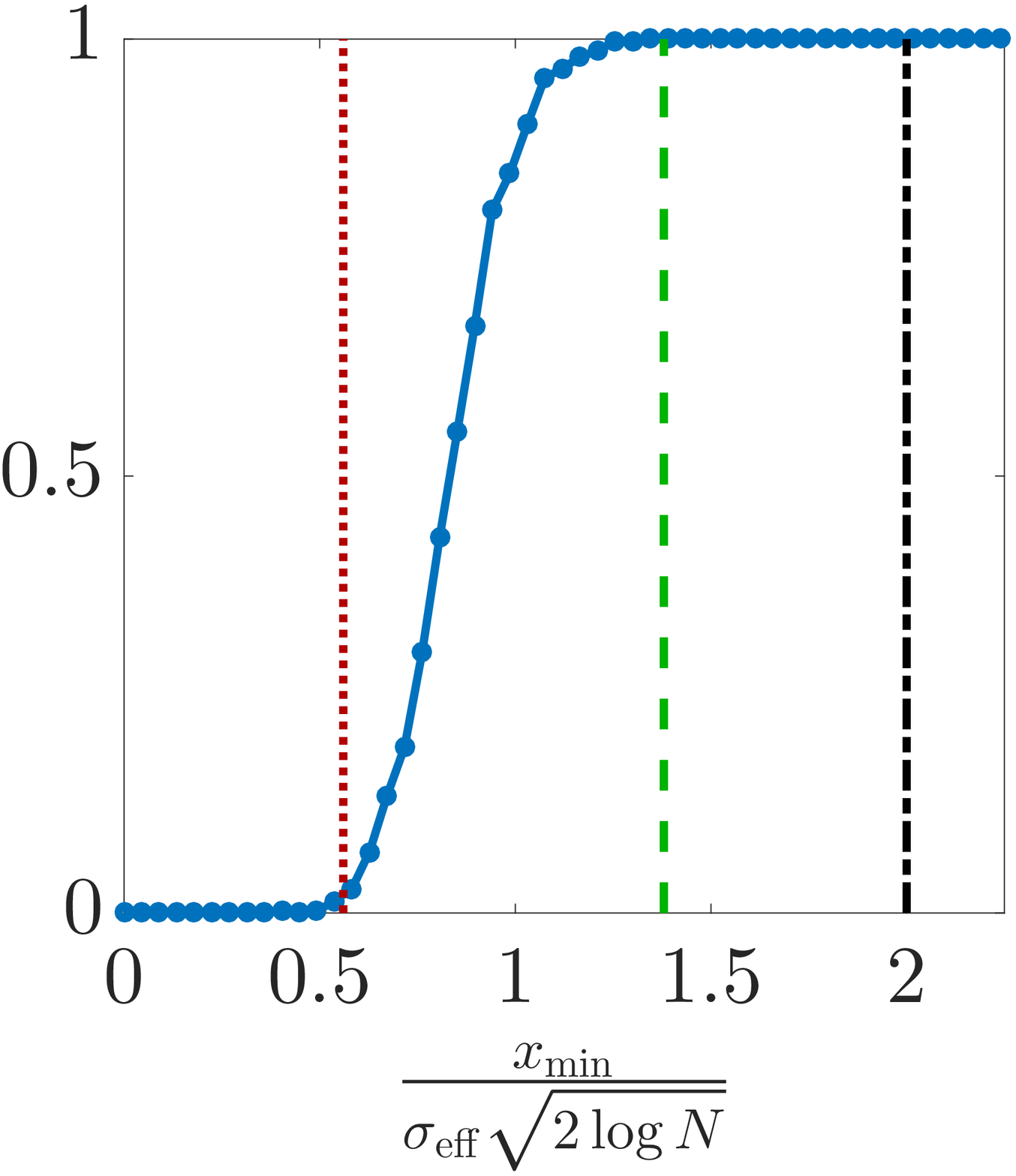}
		\caption{Sparsity $\slev=3$}
	\end{subfigure}
	\begin{subfigure}[b]{0.3\linewidth}
		\includegraphics[width=\linewidth]{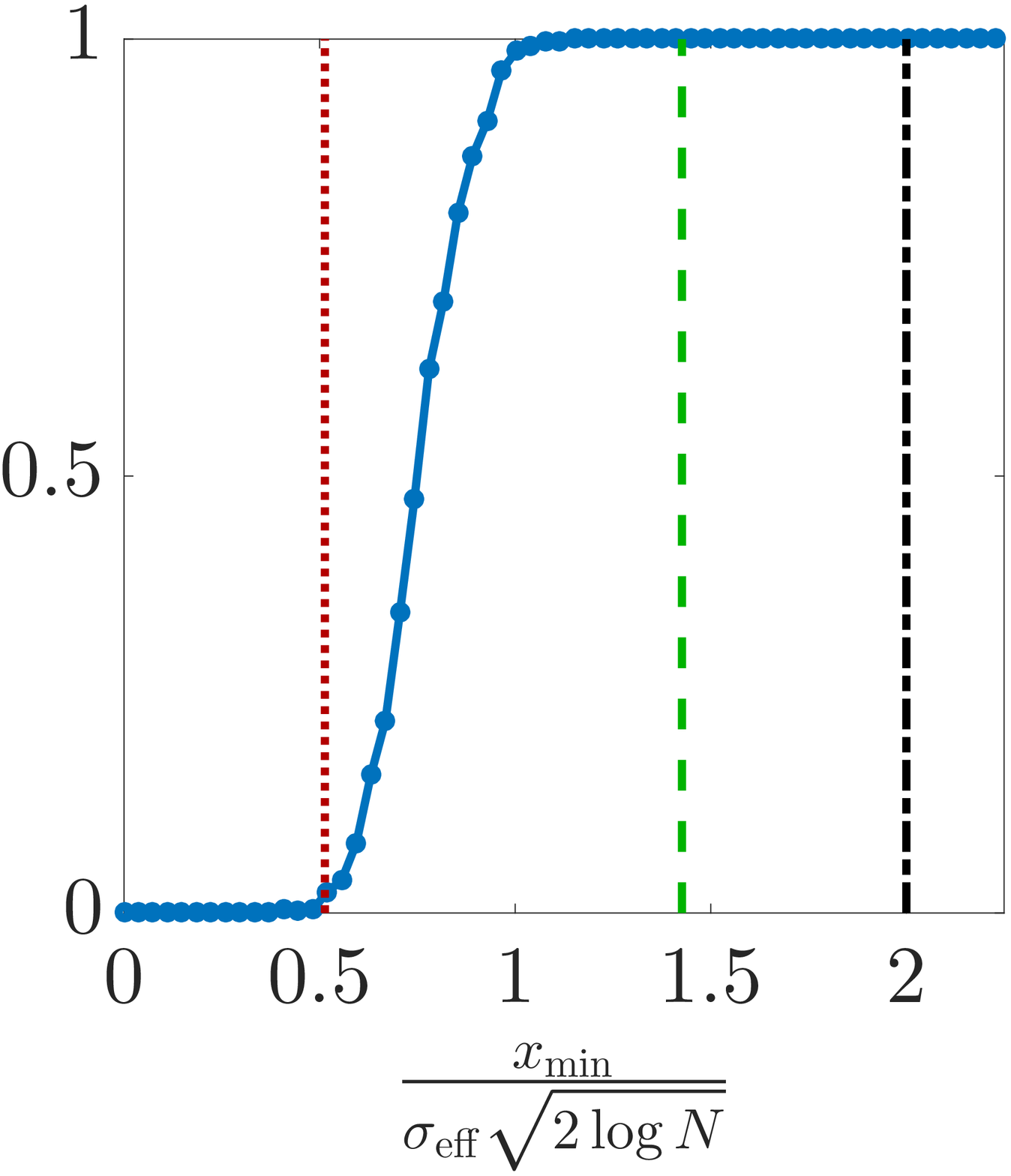}
		\caption{Sparsity $\slev=4$}
	\end{subfigure}
	\caption{The solid blue line in each figure is the empirical probability of exact support recovery of a sparse vector by OMP as a function of its normalized signal-to-noise ratio in Setting 2. The dash-dotted black line is the sufficient condition \eqref{eq:SNR_upper} by \citet{ben2010coherence}. The dashed green line is the sharper sufficient condition \eqref{eq:SNR_upper_sharp}. The dotted red line is the approximate condition \eqref{eq:apprx_lb}.
	}
	\label{fig:our_dict}
\end{figure}
We present several simulations to illustrate our sharper sufficient condition in Theorem \ref{thm:sharp_BH} and our near-tightness result in Theorem \ref{thm:LB}.
We generated $\dim\times\atnum$ dictionaries and $\slev$-sparse vectors with coefficients of equal magnitude $\nu$. For each vector $\coefv$, we drew random noise $\err\sim N\left(\bm 0,\bm I_{\dim}\right)$ with noise level $\sigma=1$ and computed the signal $\sig$ as in Eq. \eqref{eq:sig_dec}.

In Setting 1, we considered the probability of exact support recovery of $\slev$-sparse vectors with sparsity $\slev=3$ using three dictionaries of size $\left(\dim, \atnum\right) =\left( 4096,8192\right) $.
The first is a two-ortho dictionary $\dictm_1=\left[\bm{I} \: \bm{H} \right] $ composed of two orthogonal matrices -- the identity matrix and the Hadamard matrix with normalized columns. The second $\dictm_2$ is a dictionary whose atoms are drawn independently and uniformly at random from the unit sphere. For these two dictionaries the $\slev$-sparse vectors were drawn independently and uniformly at random from the $\binom{\atnum}{\slev}$ possible vectors. The third dictionary $\dictm_3$ and its corresponding $\slev$-sparse vector are the ones used to construct the near-tightness example in the proof of Theorem \ref{thm:LB} (see Eqs. \eqref{eq:dict} and \eqref{eq:lb_vector}).
Figure \ref{fig:differnt_dict} depicts the empirical probability that OMP recovered the exact support of the unknown sparse vector in Setting 1, averaged over $500$ realizations. It is interesting to note that our sufficient condition in Theorem \ref{thm:sharp_BH} indeed improves over that of \citet{ben2010coherence}. In addition, our sufficient condition is relatively sharp for small values of $\slev$. 
Another important observation is that even though condition \eqref{eq:apprx_lb} was derived considering $\dictm_3$ corresponding to the third panel, we see that the condition holds for different types of dictionaries as well.

In Setting 2, we study the probability of exact support recovery for different sparsity levels $\slev=2,3,4$ for the specific dictionary and $\slev$-sparse vector used in the proof of Theorem \ref{thm:LB}.  We constructed our dictionary of size $\left( \dim,\atnum\right) =\left( 1024,2048\right) $ with coherence $\coh=0.06$ using an alternating projection method by \citet{tropp2005designing}.
Figure \ref{fig:our_dict} depicts the empirical probability that OMP recovered the exact support of the unknown sparse vector in Setting 2, averaged over $500$ realizations.
Note that condition \eqref{eq:apprx_lb} is conservative since in our proof we analyze failure only in the first step of the algorithm. 
However, it cannot be increased much further, since the probability of recovery increases sharply at higher values of the normalized signal-to-noise ratio.
Finally, we remark that similar results are obtained for other values of $\dim,\atnum,\slev$ and $\coh$.
\section{Proofs}\label{sec:proofs}
The following auxiliary lemma will be useful in proving both Theorems \ref{thm:sharp_BH} and \ref{thm:LB}. Its proof appears in Section \ref{sec:proof_lems}.
\begin{lemma}\label{lem:abs_dep_bound}
	Let $\left(X_{1},\dots,X_{n_1}\right)\sim \atnum\left(\bm{0},\bm{\Sigma}\right)$ where $\bm{\Sigma}_{ii}=1$ for all $i\in[n_1]$. For any $\eta>0$ and $n_2\geq n_1$ the following holds
	\begin{equation}\label{eq:abs_dep_bound}
	\Pr\left[ \max_{i\in [n_1]} \abs{X_i} <\sqrt{2\eta\log n_2} \right] \geq 1-\frac{n_1}{n_2^{\eta}\sqrt{\pi\eta\log n_2}}.
	\end{equation}
\end{lemma}

\subsection{Proof of Theorem \ref{thm:sharp_BH}}\label{sec:proof_thm_sharp_BH}
The proof is based on a tighter analysis of the proof of \citep[Thm 4]{ben2010coherence}. 
First, we define "bad" random events $B_{\slev}$ and $B_{ \tilde{\atnum}}$ which indicate that the largest magnitude of inner products of the noise with support atoms and with non-support atoms is larger than their respective thresholds. We then define the "good" random event $G$ that indicates that neither $B_{\slev}$ nor $B_{ \tilde{\atnum}}$ occurs, and prove that the event $G$ occurs with probability \eqref{eq:sharp_pr}. Next, we show that under the event $G$, OMP with $\slev$ iterations successfully recovers the support of $\coefv$. Finally, we prove that OMP* with threshold $\tau=\sigma \sqrt{2(1+\alpha)\log \atnum}$ stops after exactly $\slev$ iterations, and therefore also successfully recovers the support of $\coefv$.

In details, we define the following two random events \[B_{\slev}=\left\{ \max_{k\in\Omega}\left|\left\langle \atom_{k},\sigma\err\right\rangle \right|\geq\sqrt{\beta}\tau\right\}\] and \[B_{ \tilde{\atnum}}=\left\{ \max_{i\notin\Omega}\left|\left\langle \atom_{i},\sigma\err\right\rangle \right|\geq \tau\right\},\]  
and let the random event $G=\left( B_{\slev}\cup B_{ \tilde{\atnum}}\right) ^C$ be the complement of their union.
Note that while these definitions depend on the unknown support set $\indset$, this is only for the sake of the analysis, and we do not assume that OMP receives the support $\indset$ as input.

Next, we prove that the event $G$ occurs with probability at least \eqref{eq:sharp_pr}.
Since the dictionary atoms are normalized, each random variable $\left\langle \atom_{i},\err\right\rangle $ is a standard Gaussian random variable. Therefore, applying Lemma \ref{lem:abs_dep_bound} with $n_1= \slev$, $n_2=\atnum$ and $\eta=\left(1+\alpha\right)\beta$ gives \[\Pr\left[B_{\slev}\right]\leq \frac{\slev}{\atnum^{\left(1+\alpha\right)\beta}\sqrt{\pi\left(1+\alpha\right)\beta\log \atnum}}.\]
Since  $\slev\leq \atnum^\beta$, then
\[\Pr\left[B_{\slev}\right]\leq \frac{1}{\atnum^{\alpha\beta}\sqrt{\pi\left(1+\alpha\right)\beta\log \atnum}}.\]
Similarly, we can apply Lemma \ref{lem:abs_dep_bound} again with $n_1= \tilde{\atnum}=\atnum-\slev$, $n_2=\atnum$ and $\eta=1+\alpha$ and get \[\Pr\left[B_{\tilde{\atnum}}\right]\leq \frac{\tilde{\atnum}}{\atnum^{\left(1+\alpha\right)}\sqrt{\pi\left(1+\alpha\right)\log \atnum}}.\]
Since $\tilde{\atnum} < \atnum$, then
\[\Pr\left[B_{\tilde{\atnum}}\right]\leq\frac{1}{\atnum^{\alpha}\sqrt{\pi\left(1+\alpha\right)\log \atnum}}.\] 
By the definition of $G$ and a union bound, \[\Pr\left[G\right]=\Pr\left[\left(B_{\slev}\cup B_{\tilde{\atnum}}\right)^{c}\right]\geq1-\Pr\left[B_{\slev}\right]-\Pr\left[B_{\tilde{\atnum}}\right],\]
which proves that the event $G$ occurs with probability at least \eqref{eq:sharp_pr}.

The following lemma shows that under condition \eqref{eq:SNR_upper_sharp}, one step of the OMP algorithm chooses an atom in the support $\indset$. 
\begin{lemma}\label{lem:sharp_induc_gur}
	Let $\bm{z}$ be an unknown vector with sparsity $\slev$ and support $\indset=\supp \left\lbrace \bm{z}\right\rbrace $, and let $\res = \dictm\bm{z}+\sigma\err$ where $\dictm\in \mathbb{R}^{\dim \times \atnum} $ is a dictionary with normalized columns and coherence $\coh$, and $\err\sim N\left(\bm 0,\bm I_{\dim}\right)$. Suppose that the MIP condition \eqref{eq:mip_cond} holds, that $\slev\leq\atnum^\beta$ for some $0< \beta < 1$ and that for some $\alpha\geq 0$
	\begin{equation}\label{eq:SNR_upper_sharp_for_lemma}
	\max_{i\in\indset}\abs{\bm{z}_i}\geq \sigeff\left(1+\sqrt{\beta}\right)\sqrt{2\left(1+\alpha\right)\log\atnum}.
	\end{equation}
	Then under the event $G$,
	\begin{equation}\label{eq:indu_cond}
	\max_{k\in\indset}\left|\left\langle \atom_{k},\res\right\rangle \right|>\max_{i\notin\indset}\left|\left\langle \atom_{i},\res\right\rangle \right|.
	\end{equation}
\end{lemma}
\begin{proof}[Proof of Lemma \ref{lem:sharp_induc_gur}]
	Denote by $z_{\max}=\max_{i\in\indset}\abs{\bm{z} _i}$.
	Under the event $G$, the largest magnitude of an inner product of the observed signal $\res$ with a non-support atom $i\notin \indset$ is
	\begin{eqnarray}\label{eq:M_out}
	\max_{i\notin\indset}\left|\left\langle \atom_{i},\res\right\rangle \right| & = & \max_{i\notin\indset}\left|\left\langle \atom_{i},\sigma\err\right\rangle +\sum_{j\in\indset}\bm{z}_{j}\left\langle \atom_{i},\atom_{j}\right\rangle \right|\nonumber\\
	& \leq & \max_{i\notin\indset}\left|\left\langle \atom_{i},\sigma\err\right\rangle \right|+\max_{i\notin\indset}\sum_{j\in\indset}\left|\bm{z}_{j}\left\langle \atom_{i},\atom_{j}\right\rangle \right|\nonumber\\
	& < & \tau+\slev\coh z_{\max}.
	\end{eqnarray}
	Similarly, 
	\begin{eqnarray}\label{eq:M_in}
	\max_{k\in\indset}\left|\left\langle \atom_{k},\res\right\rangle \right| & = & \max_{k\in\indset}\left|\bm{z}_{k}+\left\langle \atom_{k},\sigma\err\right\rangle +\sum_{j\in\indset\setminus\left\{ k\right\} }\bm{z}_{j}\left\langle \atom_{k},\atom_{j}\right\rangle \right|\nonumber\\
	& \geq & z_{\max}-\max_{k\in\indset}\left|\left\langle \atom_{k},\sigma\err\right\rangle +\sum_{j\in\indset\setminus\left\{ k\right\} }\bm{z}_{j}\left\langle \atom_{k},\atom_{j}\right\rangle \right|\nonumber\\
	& > & z_{\max}-\sqrt{\beta}\tau-\left(\slev-1\right)\coh z_{\max}.
	\end{eqnarray}
	Combining the last two equations gives \[\max_{k\in\indset}\left|\left\langle \atom_{k},\res\right\rangle \right|-\max_{i\notin\indset}\left|\left\langle \atom_{i},\res\right\rangle \right|>z_{\max}-\left(2m-1\right)\coh z_{\max}-\sqrt{\beta}\tau-\tau.\] 
	Substituting for $\tau$ implies that Eq. \eqref{eq:indu_cond} holds under condition \eqref{eq:SNR_upper_sharp_for_lemma}.
\end{proof}

Next, assume that $G$ occurs. We prove the first part of Theorem \ref{thm:sharp_BH} by induction. Consider the first iteration of OMP, described in Algorithm \ref{alg:OMP}. In line 3, OMP chooses an atom $\atom_i$ whose inner product with $\sig$ is maximal. In other words, condition \eqref{eq:indu_cond} must hold for $\res=\sig$ and $\bm{z}=\coefv$ for OMP to select an atom $i\in\indset$ at the first iteration. When $G$ occurs, then by condition \eqref{eq:SNR_upper_sharp} and by Lemma \ref{lem:sharp_induc_gur} OMP selects a support atom, i.e., $\hat{\indset}_1 \subseteq \indset$. Assume by induction that the set of atoms that were selected in all previous $1\leq t<\slev$ iterations is a subset of the support set, i.e., $\supp\left\lbrace \hat{\coefv}_{t}\right\rbrace=\hat{\indset}_{t} \subseteq \indset$. Hence,
\begin{equation}\label{eq:r_t}
\res_{t}=\sig-\dictm\hat{\coefv}_{t}=\dictm\left(\coefv-\hat{\coefv}_{t}\right)+\sigma\err,
\end{equation} where $\coefv-\hat{\coefv}_{t}$ is a sparse vector whose support is contained in $\indset$. 
In addition, since OMP selects exactly one atom at each iteration, \[\left|\supp\left\{ \hat{\coefv}_{t}\right\} \right|=t<\slev=\left|\supp\left\{ \coefv\right\} \right|.\] Hence, at least one entry in $\coefv-\hat{\coefv}_{t}$ is equal to its corresponding entry in $\coefv$ and
\begin{equation}\label{eq:res_min}
\max_{i\in\indset}\abs{\left(\coefv-\hat{\coefv}_{t}\right)_{i}}\geq\min_{i\in\indset}\abs{\coefv_{i}}=\xmin.
\end{equation} 
Since by Eq. \eqref{eq:SNR_upper_sharp} $\xmin$ is larger than the bound in Eq. \eqref{eq:SNR_upper_sharp_for_lemma}, we can apply Lemma \ref{lem:sharp_induc_gur} with $\res=\res_{t}$ and $\bm{z}=\coefv-\hat{\coefv}_{t}$ to conclude that under the event $G$, \[\max_{k\in\indset}\left|\left\langle \atom_{k},\res_{t}\right\rangle \right|>\max_{i\notin\indset}\left|\left\langle \atom_{i},\res_{t}\right\rangle \right|.\]
This implies that OMP chooses a support atom at iteration $t+1$. Therefore by induction the OMP algorithm recovers the unknown support of $\coefv$ under the event $G$, which concludes the proof of the first part of Theorem \ref{thm:sharp_BH}.

It remains to show that the OMP* algorithm with threshold $\tau=\sigma \sqrt{2(1+\alpha)\log \atnum}$ does not stop early before the $\slev$-th iteration, and that it does stop after the $\slev$-th iteration.
At iteration $1\leq t\leq \slev$,
\[\norm{\dictm^T\res_{t}}_{\infty} =\max_{i\in[\atnum]}\lvert\langle\atom_{i},\res_{t}\rangle\rvert \geq \max_{k\in\indset}\lvert\langle\atom_{k},\res_{t}\rangle\rvert.\]
Under the event $G$, by Eqs. \eqref{eq:M_in} and \eqref{eq:res_min}, \[\max_{k\in\indset}\lvert\langle\atom_{k},\res_{t}\rangle\rvert>\xmin\left(1-\left(\slev-1\right)\coh\right)-\sqrt{\beta}\tau. \]
Finally, by condition \eqref{eq:SNR_upper_sharp}, 
\[\norm{\dictm^T\res_{t}}_{\infty}>\frac{\left(1+\sqrt{\beta}\right)\left(1-\left(\slev-1\right)\coh\right)}{1-\left(2\slev-1\right)\coh}\tau-\sqrt{\beta}\tau=\frac{1-\left(\left(1-\sqrt{\beta}\right)\slev-1\right)\coh}{1-\left(2\slev-1\right)\coh}\tau>\tau,\]
which proves that OMP* does not stop early. 

At the end of iteration $t=\slev$ all support atoms have been selected. 
Let $\coefv_{\indset}\in\R^\slev $ and $\dictm_{\indset}\in\R^{\dim\times\slev}$ be the vector $\coefv$ and the dictionary $\dictm$ restricted to the support $\indset$ (respectively), and let $P_{\indset}=\dictm_{\indset}\dictm_{\indset}^{\dagger}=\dictm_{\indset}\left(\dictm_{\indset}^{T}\dictm_{\indset}\right)^{-1}\dictm_{\indset}^{T}$ be the projection of the observed signal onto the linear subspace spanned by the elements of $\indset$. Then \[\res_{\slev}=\sig-\dictm_{\indset}\dictm_{\indset}^{\dagger}\sig=\left(I-P_{\indset}\right)\sig=\left(I-P_{\indset}\right)\dictm_{\indset}\coefv_{\indset}+\left(I-P_{\indset}\right)\sigma\err.\]
Since $I-P_{\indset}$ is a projection to the linear space that is orthogonal to the subspace spanned by the elements of $\indset$, the first term of the last equation above is zero. Hence, under the event $G$,
\[\norm{\dictm^T\res_{\slev}}_{\infty}=\max_{i\in\left[ \atnum\right] }\lvert\langle\atom_{i},\res_{\slev}\rangle\rvert=\max_{i\in\left[ \atnum\right]}\lvert\langle\atom_{i},\left(I-P_{\indset}\right)\sigma\err\rangle\rvert\leq \max_{i\in\left[ \atnum\right]}\lvert\langle\atom_{i},\sigma\err\rangle\rvert \leq \tau.\]
Therefore, OMP* stops after exactly $\slev$ iterations. This concludes the proof of Theorem \ref{thm:sharp_BH}.
\qed

\subsection{Proof of Theorem \ref{thm:LB}}\label{sec:proof_thm_LB}
First, we present an outline of the proof. Given parameters $\atnum,\dim,\slev,\coh$ with $\dim<\atnum$, and where $\slev,\coh$ satisfy conditions \eqref{eq:sparsity_beta}-\eqref{eq:mu_cond}, we construct a dictionary $\dictm\in \R^{\dim \times \atnum} $ with coherence $\coh$ and a sparse vector $\coefv\in \R^{ \atnum}$ with sparsity $\slev$. We show that when the smallest coefficient in $\coefv$ is sufficiently small as in condition \eqref{eq:SNR_Cond}, then with probability at least $P_0$, OMP fails to detect a support atom already at the first iteration, and therefore fails to recover the support of $\coefv $.

To prove the theorem we shall use the following auxiliary lemmas.
The first lemma concerns the maximum of several correlated normal random variables.
\begin{lemma}  \label{lem:max_tail}
	Let $\left(X_{1},\dots,X_{n}\right)\sim \atnum\left(\bm{0},\bm{\Sigma}\right)$ where $\bm{\Sigma}_{ii}=1$ for all $i\in[n]$ and $ 0<\left| \bm{\Sigma}_{ij} \right| \leq \eta<1 $ for all $i\neq j\in[n]$. For $M_n = \max_{i\in [n]} X_i$, the following hold:
	\begin{enumerate}
		\item (\citep{lopes2018maximum})\textbf{.} There exists a universal constant $c_0>1$ such that 
		\begin{equation}\label{eq:lopes}
		\E\left[M_{n}\right]\geq\sqrt{2(1-\eta)\log n}-\left(c_{0}-1\right)\sqrt{\log\log n}.
		\end{equation}
		\item (\citep{tanguy2015some})\textbf{.} There exists $C>0$ such that for any $n\geq 2$ and $t>0$, 
		\begin{equation}\label{eq:tanguy}
		\Pr\left[\left|M_{n}-\E\left[M_{n}\right]\right|>t\right]\leq6e^{-Ct\sqrt{\min\left\lbrace \frac{1}{\eta},\;\log n\right\rbrace }}.
		\end{equation}
	\end{enumerate}
\end{lemma}
In constructing our specific dictionary, we will use the following lemma whose proof appears in Section \ref{sec:proof_lems}.
\begin{lemma}\label{lem:m_m_corr_mat}
	For any integer $\slev>1$ and any $\coh<\frac{1}{m-1}$, there exist vectors $\atome_{1},  \dots,  \atome_{\slev}\in \R^{\slev}$ such that for all $i,j\in[\slev]$ 
	\begin{equation}\label{eq:Gram}
	\left\langle \atome_{i},\atome_{j}\right\rangle =\begin{cases}
	1 & i=j\\
	-\coh & i\neq j.
	\end{cases}
	\end{equation}
\end{lemma}     

\begin{proof}[Proof of Theorem \ref{thm:LB}]
	Recall the notations $\tilde{\dim}=\dim-\slev$ and  $\tilde{\atnum}=\atnum-\slev$. 
	Given the sparsity $\slev$ and coherence $\coh$, we first construct vectors $\atome_{1},\dots,\atome_{\slev}\in\R^\slev$ as in Lemma \ref{lem:m_m_corr_mat}.
	Next, we construct our dictionary $\dictm=\left[ \atom_1,\dots, \atom_N\right] \in \R^{\dim \times \atnum}$ as follows
	\begin{equation}\label{eq:dict}
	\dictm=\left[\begin{array}{cccccc}
	\atome_{1} & \dots & \atome_{\slev} & \sqrt{\tmu}\bar{\atome} & \dots & \sqrt{\tmu}\bar{\atome}\\
	\boldsymbol{0} & \dots & \boldsymbol{0} & \sqrt{1-\tmu}\rr_{\slev+1} & \dots & \sqrt{1-\tmu}\rr_{\atnum}
	\end{array}\right],
	\end{equation}
	where $\bar{\atome}=\frac{\sum_{i\in\left[\slev\right]}\atome_{i}}{\left\Vert \sum_{i\in\left[\slev\right]}\atome_{i}\right\Vert }$ and the constant $\tmu$ is defined in Eq. \eqref{eq:tmu}. For future use, note that 
	\begin{equation}\label{eq:bar_norm}
	\left\Vert \sum_{i\in\left[\slev\right]}\atome_{i}\right\Vert =\sqrt{\sum_{i\in\left[\slev\right]}\sum_{i'\in\left[\slev\right]}\innerp{\atome_{i}}{\atome_{i'}}}=\sqrt{\slev\left(1-\left(\slev-1\right)\coh\right)}.
	\end{equation}
	The key requirements of the vectors $\rr_{\slev+1},\dots,\rr_N\in\R^{\tilde{\dim}}$ is that they have unit norm $\norm{\rr_{i}}=1$ and that they satisfy the following condition
	\begin{equation}\label{eq:coh_cond}
	\max_{\slev+1\leq i<j\leq \atnum}\abs{\innerp{\rr_i}{\rr_j}}\leq L.
	\end{equation}

	As the following lemma shows, condition \eqref{eq:coh_cond} implies that the coherence of $\dictm$ is $\coh$. The proof appears in Section \ref{sec:proof_lems}.
	\begin{lemma}\label{lem:coh}
		Assume that $\slev,\coh$ satisfy conditions \eqref{eq:sparsity_cond} and \eqref{eq:mu_cond}. Then, under condition \eqref{eq:coh_cond}, the coherence of the dictionary $\dictm$ of Eq. \eqref{eq:dict} is exactly $\coh$.
	\end{lemma}
	Before proceeding we remark that such a dictionary $\dictm$ indeed exists. Specifically, Lemma \ref{lem:coh_rand} in Section \ref{sec:proof_lems} shows that if $\rr_{\slev+1},\dots,\rr_N$ are drawn independently and uniformly at random from the unit sphere, and $\coh$ satisfies condition \eqref{eq:mu_cond} with a (possibly) slightly higher value $L=2\sqrt{\frac{\log{\tilde{\atnum}}}{\tilde{\dim}}}$, then condition \eqref{eq:coh_cond} holds with high probability.
	
	Let us now analyze the inability of OMP to successfully recover the support of an underlying $\slev$-sparse vector $\coefv$, given  $\sig = \dictm \coefv+\sigma\err$. Consider a dictionary $\dictm$ of the form \eqref{eq:dict} and the vector
	\begin{equation}\label{eq:lb_vector}
	\coefv=\nu\sum_{j=1}^{\slev}\bm{e}_{j},
	\end{equation}
	which implies that $\indset = \left\lbrace 1,\dots,\slev\right\rbrace $ and $\xmin=\nu$.
	Note that $\dictm\coefv=\nu\sum_{i\in\left[\slev\right]}\atom_{i}=\nu\left(\begin{array}{c}
	\sum_{i\in\left[\slev\right]}\atome_{i}\\
	\bm{0}
	\end{array}\right)$.
	From this point on we view $\dictm $ and $ \coefv$ as fixed and the randomness is only over realizations of the noise vector $\err$.
	
	Our goal is to show that if $\nu$ is sufficiently small such that condition \eqref{eq:SNR_Cond} holds, then with high probability OMP fails to recover the support  $\indset $. 
	For future use, we introduce the following two random variables that depend on the noise $\err$,
	\begin{equation}\label{eq:Aout}
	\Aout=\max_{i\notin\indset}\left|\left\langle \atom_{i},\sig\right\rangle \right|
	\end{equation} 
	and
	\begin{equation}\label{eq:Ain}
	\Ain=\max_{k\in\indset}\left|\left\langle \atom_{k},\sig\right\rangle \right|.
	\end{equation}
	A sufficient condition for the failure of OMP, as described in Algorithm \ref{alg:OMP}, is that it would choose a non-support atom in the first step of the algorithm with probability $\geq P_0$, or equivalently if  \[\Pr\left[ \Aout > \Ain \right]\geq P_0.\]
	
	As we shall see below, due to dependencies between various inner products $\left\langle \atom_{i},\sig\right\rangle$, the probability that $\Aout>\Ain$ is difficult to analyze. 
	Instead, we will introduce two other random variables $\Bout$ and $\Bin$ which satisfy $\Aout\geq\Bout$, $\Bin\geq\Ain$ and for which $\Pr\left[ \Bout > \Bin \right]$ is simpler to analyze.
	
	First, we decompose the noise into its support elements $\err_\slev\in \R^{\slev}$ and non-support elements $\err_{\tilde{\dim}}\in \R^{\tilde{\dim}}$, such that $\err=\left(\begin{array}{c}
	\err_\slev\\
	\err_{\tilde{\dim}}
	\end{array}\right).$
	Next, we analyze the random variable $\Aout$ and define the random variable $\Bout$.
	Using the value for $\tmu$ in Eq. \eqref{eq:tmu} and value of the norm in Eq. \eqref{eq:bar_norm}, the inner product of the observed signal $\sig=\dictm\coefv+\sigma\err$ with a non-support atom $i\notin \indset$ is
	\begin{eqnarray}\label{eq:inp_out}
	\left\langle \atom_{i},\sig\right\rangle  & = & \left\langle \atom_{i},\nu\sum_{j=1}^{\slev}\atom_{j}\right\rangle +\left\langle \atom_{i},\sigma\err\right\rangle\nonumber\\ &=&\left\langle \left(\begin{array}{c}
	\sqrt{\tmu}\bar{\atome}\\
	\sqrt{1-\tmu}\rr_{i}
	\end{array}\right),\nu\sum_{j=1}^{\slev}\left(\begin{array}{c}
	\atome_{j}\\
	\bm{0}
	\end{array}\right)\right\rangle +\left\langle \atom_{i},\sigma\err\right\rangle \nonumber\\
	& = & \nu\sum_{j=1}^{\slev}\left\langle \sqrt{\tmu}\bar{\atome},\atome_{j}\right\rangle +\sigma\left\langle \atom_{i},\err\right\rangle =\frac{\nu\sqrt{\tmu}}{\sqrt{\slev\left(1-\left(\slev-1\right)\coh\right)}}\sum_{j=1}^{\slev}\sum_{j'=1}^{\slev}\left\langle \atome_{j'},\atome_{j}\right\rangle +\sigma\left\langle \atom_{i},\err\right\rangle \nonumber\\
	& = & \frac{\nu\coh}{1-\left(\slev-1\right)\coh}\left(\slev\left(1-\left(\slev-1\right)\coh\right)\right)+\sigma\left\langle \atom_{i},\err\right\rangle =\nu\slev\coh+\sigma\left\langle \atom_{i},\err\right\rangle \nonumber\\
	& = & \nu\slev\coh+\sigma\sqrt{\tmu}\left\langle \bar{\atome},\err_{\slev}\right\rangle +\sigma\sqrt{1-\tmu}\left\langle \rr_{i},\err_{\tilde{\dim}}\right\rangle .
	\end{eqnarray}
	We define $\Bout$ by
	\begin{equation}\label{eq:Bout}
	\Bout=\nu\slev\coh-\sigma\sqrt{\tmu}\left|\left\langle \bar{\atome},\err_{\slev}\right\rangle \right|+\sigma\sqrt{1-\tmu}\max_{i\notin\indset}\left\langle \rr_{i},\err_{\tilde{\dim}}\right\rangle .
	\end{equation}
	Using Eq. \eqref{eq:Aout} and \eqref{eq:inp_out}, 
	\begin{eqnarray*}
		\Aout & = & \max_{i\notin\indset}\left|\left\langle \atom_{i},\sig\right\rangle \right|\geq\max_{i\notin\indset}\left\langle \atom_{i},\sig\right\rangle \\
		& = & \nu\slev\coh+\sigma\sqrt{\tmu}\left\langle \bar{\atome},\err_{\slev}\right\rangle +\sigma\sqrt{1-\tmu}\max_{i\notin\indset}\left\langle \rr_{i},\err_{\tilde{\dim}}\right\rangle \\
		& \geq & \nu\slev\coh-\sigma\sqrt{\tmu}\left|\left\langle \bar{\atome},\err_{\slev}\right\rangle \right|+\sigma\sqrt{1-\tmu}\max_{i\notin\indset}\left\langle \rr_{i},\err_{\tilde{\dim}}\right\rangle =\Bout.
	\end{eqnarray*}
	
	We now analyze the random variable $\Ain$ and define the random variable $\Bin$.
	The inner product of the observed signal $\sig$ with a support atom $k\in\indset$ is 
	\begin{eqnarray}\label{eq:inp_in}
	\left\langle \atom_{k},\sig\right\rangle  & = & \left\langle \atom_{k},\nu\sum_{j=1}^{\slev}\atom_{j}\right\rangle +\left\langle \atom_{k},\sigma\err\right\rangle 
	\nonumber\\
	& = & \nu\sum_{j=1}^{\slev}\left\langle \atome_{k},\atome_{j}\right\rangle +\sigma\left\langle \atom_{k},\err\right\rangle =\nu\left(1-\left(\slev-1\right)\coh\right)+\sigma\left\langle \atome_{k},\err_{\slev}\right\rangle .
	\end{eqnarray}

	To circumvent the dependence between the random variables $\left\langle \atome_{k},\err_\slev\right\rangle$ and $\left\langle \bar{\atome},\err_\slev\right\rangle$, we decompose each $\atome_{k}$ into two components, $\atome_k^{\parallel}$ which is parallel to $\bar{\atome}$ and $\atomeh_{k}$ which is orthogonal to $\bar{\atome}$. Using Eq. \eqref{eq:eps} and \eqref{eq:bar_norm}, for each $k\in\indset$ 
	\begin{equation}
	\left\langle \atome_{k},\bar{\atome}\right\rangle       =       \frac{1}{\left\Vert \sum_{i\in\left[\slev\right]}\atome_{i}\right\Vert }\sum_{i\in\left[\slev\right]}\left\langle \atome_{k},\atome_{i}\right\rangle =\frac{1-\left(\slev-1\right)\coh}{\sqrt{\slev\left(1-(\slev-1)\coh\right)}}=\rho.
	\end{equation}
	Thus $\atomeh_{k}=\atome_{k}- \left\langle \atome_{k},\bar{\atome}\right\rangle\bar{\atome}=\atome_{k}-\rho\bar{\atome}. $
	Combining this relation with Eq. \eqref{eq:inp_in}, we can rewrite $\Ain$ as \[\Ain =\max_{k\in\indset}\left|\nu\left(1-\left(\slev-1\right)\coh\right)+\sigma\left\langle \atomeh_{k}+\rho\bar{\atome},\err_{\slev}\right\rangle \right|. \]
	We define $\Bin$ by 
	\begin{equation}
	\Bin=\nu\left(1-\left(\slev-1\right)\coh\right)+\sigma\rho\left|\left\langle \bar{\atome},\err_{\slev}\right\rangle \right|+\sigma\max_{k\in\indset}\left|\left\langle \atomeh_{k},\err_{\slev}\right\rangle \right|.
	\end{equation}
	By the triangle inequality $\Ain\leq\Bin$.
	
	Now that we defined $\Bout$ and $\Bin$, we proceed to analyze the probability that $\Bout>\Bin$, or equivalently,
	\begin{equation}\label{eq:B_ineq}
	\sqrt{1-\tmu}\max_{i\notin\indset}\left\langle \rr_{i},\err_{\tilde{\dim}}\right\rangle >\frac{\nu}{\sigeff}+\left(\rho+\sqrt{\tmu}\right)\left|\left\langle \bar{\atome},\err_{\slev}\right\rangle \right|+\max_{k\in\indset}\left|\left\langle \atomeh_{k},\err_{\slev}\right\rangle \right|.
	\end{equation}
	For constants $b_1,b_2$ that will be determined later, denote the following three probabilities \[P_1=\Pr\left[\max_{i\notin\indset}\left\langle \rr_{i},\err_{\tilde{\dim}}\right\rangle >b_{1}\right],\]
	\[P_2=\Pr\left[\left|\left\langle \bar{\atome},\err_{\slev}\right\rangle \right|<b_{2}\right],\] and \[P_3=\Pr\left[\max_{k\in\indset}\left|\left\langle \atomeh_{k},\err_{\slev}\right\rangle \right|<\sqrt{1-\tmu}b_{1}-\frac{\nu}{\sigeff}-\left(\rho+\sqrt{\tmu}\right)b_{2}\right].\]
	By the statistical independence of $\err_\slev$ and $\err_{\tilde{\dim}}$ and the linear independence of $\bar{\atome}$ and $\atomeh_{k}$ for all $k\in\indset$, 
	\begin{eqnarray*}
		\Pr\left[\Bout>\Bin\right] & \geq & \Pr\left[\max_{i\notin\indset}\left\langle \rr_{i},\err_{\tilde{\dim}}\right\rangle >b_{1}\right]\times\Pr\left[\left|\left\langle \bar{\atome},\err_{\slev}\right\rangle \right|<b_{2}\right]\\
		&  & \times\Pr\left[\max_{k\in\indset}\left|\left\langle \atomeh_{k},\err_{\slev}\right\rangle \right|<\sqrt{1-\tmu}b_{1}-\frac{\nu}{\sigeff}-\left(\rho+\sqrt{\tmu}\right)b_{2}\right]\\
		& = & P_{1}\cdot P_{2}\cdot P_{3}.
	\end{eqnarray*}
	Hence, instead of proving that \eqref{eq:B_ineq} holds with probability at least $P_0$, it suffices to prove that $P_{1}\cdot P_{2}\cdot P_{3}\geq P_0.$

	We proceed by calculating each of these probabilities, beginning with $P_1$. Since $\rr_{\slev+1},\dots,\rr_N$ are fixed unit vectors, each inner product between $\rr_i$ and the vector of standard normals $\err_{\tilde{\dim}}$ is a standard normal random variable. 
	Let \[b_1=\sqrt{2\left(1-\coh\right)\log\tilde{\atnum}}-c_{0}\sqrt{\log\log\tilde{\atnum}},\]
	where $c_0$ is the constant from Lemma \ref{lem:max_tail}.
	Denote by $M_{\tilde{\atnum}}=\max_{i\notin\indset}\left\langle \rr_{i},\err_{\tilde{\dim}}\right\rangle$. By the first part of Lemma \ref{lem:max_tail}, 
	\begin{eqnarray*}
		\E\left[M_{\tilde{\atnum}}\right] & \geq & \sqrt{2(1-\mu)\log\tilde{\atnum}}-(c_{0}-1)\sqrt{\log\log\tilde{\atnum}}\\
		& = & b_{1}+\sqrt{\log\log\tilde{\atnum}}.
	\end{eqnarray*}
	Therefore, by the triangle inequality,
	\begin{eqnarray*}
		P_{1} & = & \Pr\left[M_{\tilde{\atnum}}>b_{1}\right]\geq\Pr\left[M_{\tilde{\atnum}}>\E\left[M_{\tilde{\atnum}}\right]-\sqrt{\log\log\tilde{\atnum}}\right]\\
		&\geq  & \Pr\left[\left|M_{\tilde{\atnum}}-\E\left[M_{\tilde{\atnum}}\right]\right|<\sqrt{\log\log\tilde{\atnum}}\right].
	\end{eqnarray*}
	By the second part of Lemma \ref{lem:max_tail}, 
	\begin{equation*}
	P_1  \geq  1-6e^{-C\sqrt{\log\log\tilde{\atnum}\min\left\lbrace \coh^{-1},\;\log\tilde{\atnum}\right\rbrace }}.
	\end{equation*}

	Next, we calculate $P_2$. Let $b_2=\sqrt{2\log\log\atnum}$. The term $\left\langle \bar{\atome},\err_\slev\right\rangle $ is simply a standard normal variable. 
	By Lemma \ref{lem:abs_dep_bound} with parameters $n_1=1$, $n_2=\log\atnum$ and $\eta=1$, we obtain that
	\[P_{2}=\Pr\left[\left|\left\langle \bar{\atome},\err_{\slev}\right\rangle \right|<\sqrt{2\log\log\atnum}\right]>1-\frac{1}{\log\atnum\sqrt{\pi\log\log\atnum}}.\]
	
	Lastly, we calculate $P_3$. 
	Recall that by construction $\xmin=\nu$. By Eq. \eqref{eq:SNR_Cond},
	\[\frac{\nu}{\sigeff}=\sqrt{1-\tmu}b_{1}-\left(\rho+\sqrt{\tmu}\right)b_{2}-\sqrt{2\beta\left(1-\rho^{2}\right)\log\atnum}.\]
	Therefore,
	\begin{eqnarray*}
		P_{3} & = & \Pr\left[\max_{k\in\indset}\left|\left\langle \atomeh_{k},\err_{\slev}\right\rangle \right|<\sqrt{1-\tmu}b_{1}-\frac{\nu}{\sigeff}-\left(\rho+\sqrt{\tmu}\right)b_{2}\right]\\
		& = & \Pr\left[\max_{k\in\indset}\left|\left\langle \atomeh_{k},\err_{\slev}\right\rangle \right|<\sqrt{2\beta\left(1-\rho^{2}\right)\log\atnum}\right]\\
		& = & \Pr\left[\max_{k\in\Omega}\left|\frac{\left\langle \atomeh_{k},\err_{\slev}\right\rangle }{\sqrt{1-\rho^{2}}}\right|<\sqrt{2\beta\log\atnum}\right].
	\end{eqnarray*}
	Note that for all $k\in \indset$, $\left\Vert \atomeh_{k}\right\Vert =\sqrt{1-\rho^2}$. Hence, each random variable $\frac{\left\langle \atomeh_{k},\err_\slev\right\rangle}{\sqrt{1-\rho^2}} $ is Gaussian with zero mean and variance $1$. We can apply Lemma \ref{lem:abs_dep_bound} with $n_1= \slev$, $n_2=\atnum$ and $\eta=\beta$, and use the inequality \eqref{eq:sparsity_beta} to get
	\[P_{3}\geq1-\frac{\slev}{\atnum^{\beta}\sqrt{\pi\beta\log \atnum}}\geq 1-\frac{1}{\sqrt{\pi\beta\log \atnum}}.\]
	
	By a union bound, 
	for sufficiently large $\dim$ and $\atnum$ the probability that OMP fails to recover the support $\Omega$ is at least $P_0$, which completes the proof of Theorem \ref{thm:LB}.
\end{proof}

\subsection{Proofs of Lemmas} \label{sec:proof_lems}
To conclude, we prove the auxiliary lemmas.

\begin{proof}[Proof of Lemma \ref{lem:abs_dep_bound}]   
	The proof is similar to that of \citep[Lemma 2]{ben2010coherence}.
	By \citet[Thm. 1]{vsidak1967rectangular}, since $X_{1},\dots,X_{n_1}$ are jointly Gaussian random variables, then
	\begin{equation}\label{eq:sidak}
	\Pr\left[ \max_{i\in [n_1]} \abs{X_i} <\sqrt{2\eta\log n_2} \right]\geq \prod_{i\in [n_1]} \Pr\left[ \abs{X_i} \leq\sqrt{2\eta\log n_2} \right]=\Pr\left[ \abs{X_1} \leq\sqrt{2\eta\log n_2} \right]^{n_1}.
	\end{equation}
	Each $X_i$ is a standard normal random variable. Therefore, \[\Pr\left[ \abs{X_1} \leq x \right]=1-2Q\left( x\right), \]
	where $Q(x)$ is the Gaussian tail probability function. Applying the inequality \[Q\left(x\right)\leq\frac{1}{x\sqrt{2\pi}}e^{-\frac{x^{2}}{2}},\] with $x=\sqrt{2\eta\log n_2}$ gives
	\begin{equation}\label{eq:single_y1}
	\Pr\left[ \abs{X_1} \leq\sqrt{2\eta\log n_2} \right]\geq 1-\frac{e^{-\eta\log n_2}}{\sqrt{\pi\eta\log n_2}}=1-\frac{1}{n_2^{\eta}\sqrt{\pi\eta\log n_2}}.
	\end{equation}
	Inserting Eq. \eqref{eq:single_y1} into \eqref{eq:sidak} and using the inequality $\left( 1-a\right) ^n \geq 1-an$ completes the proof.
\end{proof}

\begin{proof}[Proof of Lemma \ref{lem:m_m_corr_mat}]
	Let $G\in\R^{\slev\times\slev}$ be the following symmetric matrix with entries \[G_{ij} =\begin{cases}
	1 & i=j\\
	-\coh & i\neq j.
	\end{cases}\]
	Hence, $G$ can be rewritten as a rank-one perturbation of the identity matrix \[G=-\coh\bm{1}\bm{1}^{T}+\left(1+\coh\right)\bm{I}.\]
	If $\coh<\frac{1}{m-1}$, then $G$ is positive definite. Therefore, it is the Gram matrix of a set of linearly independent vectors, i.e., there exist $\atome_{1}, \dots, \atome_{\slev}$ such that condition \eqref{eq:Gram} holds, which completes the proof \citep[p. 441]{horn2012matrix}.
	
	For completeness, we describe an explicit construction. Let $\V=\left[\vv_{1}\dots\vv_{\slev}\right]$ be an orthogonal matrix where $\vv_1=\frac{1}{\sqrt{\slev}}\bm{1}$ and
	\begin{equation*}
	\Y=\left(\begin{array}{c}
	\sqrt{\frac{1-\left(\slev-1\right)\coh}{\slev}}\bm{1}^{T}\\
	\sqrt{1+\coh}\vv_{2}^{T}\\
	\vdots\\
	\sqrt{1+\coh}\vv_{\slev}^{T}
	\end{array}\right).
	\end{equation*}
	
	Let us now prove that $\Y$ indeed satisfies condition \eqref{eq:Gram}.
	Since $\V$ is orthogonal, its rows also form an orthonormal basis of $\R^\slev$.
	First, consider the diagonal entries of the Gram matrix $\Y^{T}\Y$. For all $i\in [\slev]$, $\frac{1-\left(\slev-1\right)\coh}{\slev}=-\coh+\left(1+\coh\right)\frac{1}{\slev}=-\coh+\left(1+\coh\right)\V_{1i}^{2}$, and therefore
	\begin{equation*}
	(\Y^{T}\Y)_{ii}  =  \frac{1-\left(\slev-1\right)\coh}{\slev}+\left(1+\coh\right)\sum_{k=2}^{\slev}\V_{ki}^{2}
	=  -\coh+\left(1+\coh\right)\sum_{k=1}^{\slev}\V_{ki}^{2}=1.
	\end{equation*}
	Similarly, for all $i\neq j \in [\slev]$, $\frac{1-\left(\slev-1\right)\coh}{\slev}=-\coh+\left(1+\coh\right)\frac{1}{\slev}=-\coh+\left(1+\coh\right)\V_{1i}\V_{1j}$, and therefore
	\begin{equation*}
	(\Y^{T}\Y)_{ij}  =  \frac{1-\left(\slev-1\right)\coh}{\slev}+\left(1+\coh\right)\sum_{k=2}^{\slev}\V_{ki}\V_{kj}
	=  -\coh+\left(1+\coh\right)\sum_{k=1}^{\slev}\V_{ki}\V_{kj}
	=  -\coh.
	\end{equation*}         
\end{proof}

\begin{proof}[Proof of Lemma \ref{lem:coh}]
	To prove that the coherence of $\dictm$ is $\coh$ we need to analyze three types of dot products $\left\langle \atom_{i},\atom_{j}\right\rangle$. The first type is $1\leq i<j\leq \slev$, the second type is $1\leq i\leq \slev<j\leq \atnum$, and the third is $\slev+1\leq i<j\leq \atnum$.
	
	Beginning with the first type, by construction, for any $1\leq i<j\leq \slev$, \[\abs{\innerp{\atom_i}{\atom_j}}=\abs{\innerp{\atome_i}{\atome_j}}=\coh.\]
	For the second type, by Eq. \eqref{eq:tmu}, for any $1\leq i\leq \slev<j\leq \atnum$,
	\[\abs{\left\langle \atom_{i},\atom_{j}\right\rangle}  = \abs{\left\langle \atome_{i},\sqrt{\tmu}\bar{\atome}\right\rangle} =\frac{\coh}{\rho}\frac{1}{\left\Vert \sum_{i'\in\left[\slev\right]}\atome_{i'}\right\Vert }\abs{\sum_{i'\in\left[\slev\right]}\left\langle \atome_{i},\atome_{i'}\right\rangle}. \]
	Inserting Eq. \eqref{eq:eps} and \eqref{eq:bar_norm},
	\[\abs{\left\langle \atom_{i},\atom_{j}\right\rangle}=\coh\sqrt{\frac{\slev}{1-(\slev-1)\coh}}\frac{1-(\slev-1)\coh}{\sqrt{\slev\left(1-(\slev-1)\coh\right)}}=\coh.\]
	Finally, we address the third type. By the triangle inequality and condition \eqref{eq:coh_cond},
	\begin{equation*}
	\abs{\left\langle \atom_{i},\atom_{j}\right\rangle }=\abs{\tmu+\left(1-\tmu\right)\left\langle \rr_{i},\rr_{j}\right\rangle }\leq\tmu+\left(1-\tmu\right)\left|\left\langle \rr_{i},\rr_{j}\right\rangle \right|\leq\tmu+\left(1-\tmu\right)L.
	\end{equation*}
	It remains to show that for values of $\coh$ in the range of Eq. \eqref{eq:mu_cond},
	\begin{equation}\label{eq:innerprod_cond}
	\tmu+\left(1-\tmu\right)L\leq\mu.
	\end{equation}
	Using the definition \eqref{eq:tmu} of $\tmu$, condition \eqref{eq:innerprod_cond} is \[\frac{\slev\coh^{2}}{1-(\slev-1)\coh}+\left(1-\frac{\slev\coh^{2}}{1-(\slev-1)\coh}\right)L\leq\coh.\] In turn, this inequality can be rewritten as the following quadratic equation \[\coh^{2}\left(2\slev-1-Lm\right)-\coh\left(L\left(\slev-1\right)+1\right)+L\leq0.\]
	Notice that since $L<1$, the term $2\slev-1-Lm>\slev-1\geq 0$. The above inequality is thus satisfied by values of $\coh$ in Eq. \eqref{eq:mu_cond}.
	This range is not empty if $\frac{4L\left(2\slev-1-Lm\right)}{\left(L\left(\slev-1\right)+1\right)^{2}} \leq 1$. It is easy to verify that this condition holds for $\slev$ values in \eqref{eq:sparsity_cond}. Note that the above condition also holds for $\slev\geq\frac{3-L+\sqrt{8-8L}}{L},$ however this range is often not possible due to the MIP condition \eqref{eq:mip_cond}.
\end{proof}

\begin{lemma}\label{lem:coh_rand}
	Let $\rr_{\slev+1},\dots,\rr_N$ be $\tilde{\atnum}$ vectors drawn independently and uniformly at random from the $\tilde{\dim}$-dimensional unit sphere. Suppose $\tilde{\atnum}=\tilde{\atnum}_{\tilde{\dim}}\rightarrow \infty$ satisfies $\frac{\log  \tilde{\atnum}}{\tilde{\dim}}\rightarrow 0$ as $\tilde{\dim}\rightarrow\infty$. Then as $\tilde{\dim}\rightarrow\infty$, condition \eqref{eq:coh_cond} with $L=2\sqrt{\frac{\log{\tilde{\atnum}}}{\tilde{\dim}}}$ holds with probability $e^{-1/\sqrt{8\pi\log\tilde{\atnum}}}$.
\end{lemma}
To prove Lemma \ref{lem:coh_rand}, we need the following auxiliary lemma which bounds the largest magnitude of an inner product between random unit vectors.
\begin{lemma} [\citep{cai2012phase}]\label{lem:rand_coh_bound}
	Let $\bm a_1,\dots \bm a_N$ be i.i.d. vectors drawn uniformly at random from the $\dim$-dimensional unit sphere and let \[L_{\dim}=\max_{1\leq i<j\leq \atnum}\left|\left\langle \bm{a}_{i},\bm{a}_{j}\right\rangle \right|. \] Suppose $\atnum=N_p\rightarrow \infty$ satisfies $\frac{\log  \atnum}{\dim}\rightarrow 0$ as $\dim\rightarrow\infty$. Then as $\dim\rightarrow\infty$, the random variable \[ \dim\log(1-L_p^2)+4\log \atnum-\log\log \atnum\] converges weakly to an extreme value distribution with the distribution function $F(y)=1-e^{-Ke^{y/2}}$ for $y\in\R$ and $K=\frac{1}{\sqrt{8\pi }}$.
\end{lemma}
\begin{proof}[Proof of Lemma \ref{lem:coh_rand}]        
	Note that in the regime stated in the lemma, $L_{\tilde{\dim}}\rightarrow 0 $ as $\tilde{\dim}\rightarrow\infty$. Hence $\log(1-L_{\tilde{\dim}}^{2})\approx-L_{\tilde{\dim}}^{2}$. By Lemma \ref{lem:rand_coh_bound} for $y=-\log\log (\tilde{\atnum})$ and $L=2\sqrt{\frac{\log{\tilde{\atnum}}}{\tilde{\dim}}}$, in the limit
	\begin{eqnarray*}
		\Pr\left[\tilde{\dim}\log(1-L_{\tilde{\dim}}^{2})+4\log\tilde{\atnum}-\log\log\tilde{\atnum}\geq-\log\log\tilde{\atnum}\right] & = & \Pr\left[L_{\tilde{\dim}}\leq L\right]\\
		& = & e^{-Ke^{-\log\log\tilde{\atnum}/2}}\\
		& = & e^{-1/\sqrt{8\pi\log\tilde{\atnum}}}.
	\end{eqnarray*}
	Therefore as $\tilde{\dim}\rightarrow\infty$, $\max_{\slev+1\leq i<j\leq\atnum}\left|\left\langle \rr_{i},\rr_{j}\right\rangle \right|=L_{\tilde{\dim}}\leq L$ and condition \eqref{eq:coh_cond} is satisfied with probability $e^{-1/\sqrt{8\pi\log\left(\tilde{\atnum}\right)}}$.
\end{proof}

\section*{Funding}
R.K. was partially supported by ONR Award N00014-18-1-2364, the Israel Science Foundation grant \#1086/18, and a Minerva Foundation grant.

\bibliographystyle{plainnat}

\end{document}